%% file: front.tex
\documentclass[11pt]{amsart}

\input{header} 

\begin{document}
\thispagestyle{empty}

\begin{abstract} 
We classify all simple unitarizable supermodules over special linear Lie superalgebras using the algebraic Dirac operator introduced by Huang and Pand\v{z}i\'c and the associated Dirac inequalities. The same argument treats finite-dimensional and infinite-dimensional supermodules without requiring explicit realizations or complete branching rules. 
\end{abstract}

\maketitle

\setlength{\parindent}{1em}

\input{main} 

\bibliography{literatur}
\bibliographystyle{alpha}
\end{document}

%% file: header.tex
\usepackage{marginnote}
\usepackage{anysize} 
\usepackage{comment}
\marginsize{2cm}{2cm}{2.3cm}{2.3cm}
\usepackage{comment}
\usepackage{listings}
\usepackage{xcolor}
\usepackage{amsmath}
\usepackage{mathtools}
\usepackage{booktabs}
\usepackage[all]{xy}
\usepackage[utf8]{inputenc}
\usepackage{varioref}
\usepackage{upgreek}
\usepackage{amsfonts}
\usepackage{amssymb}
\usepackage{bbm}
\usepackage{esint}
\usepackage{graphicx}
\usepackage{subcaption}
\usepackage{tikz}
\usepackage{empheq}
\usepackage{enumitem}
\usepackage{tikz-cd}
\usepackage{upgreek}
\usepackage{todonotes}
\usetikzlibrary{matrix,arrows,decorations.pathmorphing}
\usepackage{mathrsfs}
\usepackage[hypertexnames=false,backref=page,pdftex,
 	pdfpagemode=UseNone,
 	breaklinks=true,
 	extension=pdf,
 	colorlinks=true,
 	linkcolor=blue,
 	citecolor=red,
 	urlcolor=blue,
 ]{hyperref}

\newcommand{\HH}{{\mathbb H}}

\renewcommand{\L}{{\mathbb L}}

\newcommand{\R}{{\mathbb R}}

\newcommand{\tr}{\operatorname{tr}}

\newcommand{\str}{\operatorname{str}}

\newcommand{\End}{\operatorname{End}}

\newcommand{\Dirac}{\operatorname{D}}

\def\L{{\mathfrak{L}}}
\def\R{{\mathfrak{R}}}
\def\WW{{\mathcal{W}}}

\def\dd{{\mathfrak{d}}}

\def\CC{{\mathbb C}}
\def\RR{{\mathbb R}}

\def\su{{\mathfrak{su}}}

\def\sl{{\mathfrak{sl}}}
\def\st{{\text{st}}}
\def\nst{{\text{nst}}}

\def\HH{{\mathcal{H}}}

\def\hh{{\mathfrak h}}

\def\pp{{\mathfrak p}}
\def\uu{{\mathfrak u}}

\def\kk{{\mathfrak k}}
\def\u{{\mathfrak u}}
\def\ll{{\mathfrak l}}

\def\bb{{\mathfrak b}}
\def\even{{\mathfrak{g}_{\bar{0}}}}
\def\odd{{\mathfrak{g}_{\bar{1}}}}
\def\pp{{\mathfrak p}}

\def\nn {{\mathfrak{n}}}
\def\pp{{\mathfrak{p}}}
\def\gg{{\mathfrak{g}}}

\def\ZZ{{\mathbb Z}}

\def\Weyl{{\mathscr{W}(\odd)}}

\def\slmn{{\mathfrak{sl}(m\vert n)}}

\def\glmn{{\mathfrak{gl}(m\vert n)}} 
\def\UE{{\mathfrak U}}

\makeatletter
\newcommand*{\da@rightarrow}{\mathchar"0\hexnumber@\symAMSa 4B }
\newcommand*{\da@leftarrow}{\mathchar"0\hexnumber@\symAMSa 4C }
\newcommand*{\xdashrightarrow}[2][]{%
  \mathrel{%
    \mathpalette{\da@xarrow{#1}{#2}{}\da@rightarrow{\,}{}}{}%
  }%
}
\newcommand{\xdashleftarrow}[2][]{%
  \mathrel{%
    \mathpalette{\da@xarrow{#1}{#2}\da@leftarrow{}{}{\,}}{}%
  }%
}
\newcommand*{\da@xarrow}[7]{%
  \sbox0{$\ifx#7\scriptstyle\scriptscriptstyle\else\scriptstyle\fi#5#1#6\m@th$}%
  \sbox2{$\ifx#7\scriptstyle\scriptscriptstyle\else\scriptstyle\fi#5#2#6\m@th$}%
  \sbox4{$#7\dabar@\m@th$}%
  \dimen@=\wd0 %
  \ifdim\wd2 >\dimen@
    \dimen@=\wd2 %
  \fi
  \count@=2 %
  \def\da@bars{\dabar@\dabar@}%
  \@whiledim\count@\wd4<\dimen@\do{%
    \advance\count@\@ne
    \expandafter\def\expandafter\da@bars\expandafter{%
      \da@bars
      \dabar@ 
    }%
  }%
  \mathrel{#3}%
  \mathrel{%
    \mathop{\da@bars}\limits
    \ifx\\#1\\%
    \else
      _{\copy0}%
    \fi
    \ifx\\#2\\%
    \else
      ^{\copy2}%
    \fi
  }%
  \mathrel{#4}%
}
\makeatother

\makeatletter
\newsavebox\myboxA
\newsavebox\myboxB
\newlength\mylenA

\newcommand*\xtilde[2][0.8]{%
    \sbox{\myboxA}{$\m@th#2$}%
    \setbox\myboxB\null
    \ht\myboxB=\ht\myboxA%
    \dp\myboxB=\dp\myboxA%
    \wd\myboxB=#1\wd\myboxA
    \sbox\myboxB{$\m@th\widetilde{\copy\myboxB}$}
    \setlength\mylenA{\the\wd\myboxA}
    \addtolength\mylenA{-\the\wd\myboxB}%
    \ifdim\wd\myboxB<\wd\myboxA%
       \rlap{\hskip 0.5\mylenA\usebox\myboxB}{\usebox\myboxA}%
    \else
        \hskip -0.5\mylenA\rlap{\usebox\myboxA}{\hskip 0.5\mylenA\usebox\myboxB}%
    \fi}

\newbox\usefulbox

\def\getslant #1{\strip@pt\fontdimen1 #1}

\def\xxtilde #1{\mathchoice
 {{\setbox\usefulbox=\hbox{$\m@th\displaystyle #1$}%
    \dimen@ \getslant\the\textfont\symletters \ht\usefulbox
    \divide\dimen@ \tw@ 
    \kern\dimen@ 
    \xtilde{\kern-\dimen@ \box\usefulbox\kern\dimen@ }\kern-\dimen@ }}
 {{\setbox\usefulbox=\hbox{$\m@th\textstyle #1$}%
    \dimen@ \getslant\the\textfont\symletters \ht\usefulbox
    \divide\dimen@ \tw@ 
    \kern\dimen@ 
    \xtilde{\kern-\dimen@ \box\usefulbox\kern\dimen@ }\kern-\dimen@ }}
 {{\setbox\usefulbox=\hbox{$\m@th\scriptstyle #1$}%
    \dimen@ \getslant\the\scriptfont\symletters \ht\usefulbox
    \divide\dimen@ \tw@ 
    \kern\dimen@ 
    \xtilde{\kern-\dimen@ \box\usefulbox\kern\dimen@ }\kern-\dimen@ }}
 {{\setbox\usefulbox=\hbox{$\m@th\scriptscriptstyle #1$}%
    \dimen@ \getslant\the\scriptscriptfont\symletters \ht\usefulbox
    \divide\dimen@ \tw@ 
    \kern\dimen@ 
    \xtilde{\kern-\dimen@ \box\usefulbox\kern\dimen@ }\kern-\dimen@ }}%
 {}}

\newcommand*\xoverline[2][0.75]{%
    \sbox{\myboxA}{$\m@th#2$}%
    \setbox\myboxB\null
    \ht\myboxB=\ht\myboxA%
    \dp\myboxB=\dp\myboxA%
    \wd\myboxB=#1\wd\myboxA
    \sbox\myboxB{$\m@th\overline{\copy\myboxB}$}
    \setlength\mylenA{\the\wd\myboxA}
    \addtolength\mylenA{-\the\wd\myboxB}%
    \ifdim\wd\myboxB<\wd\myboxA%
       \rlap{\hskip 0.5\mylenA\usebox\myboxB}{\usebox\myboxA}%
    \else
        \hskip -0.5\mylenA\rlap{\usebox\myboxA}{\hskip 0.5\mylenA\usebox\myboxB}%
    \fi}

\def\xxoverline #1{\mathchoice
 {{\setbox\usefulbox=\hbox{$\m@th\displaystyle #1$}%
    \dimen@ \getslant\the\textfont\symletters \ht\usefulbox
    \divide\dimen@ \tw@ 
    \kern\dimen@ 
    \overline{\kern-\dimen@ \box\usefulbox\kern\dimen@ }\kern-\dimen@ }}
 {{\setbox\usefulbox=\hbox{$\m@th\textstyle #1$}%
    \dimen@ \getslant\the\textfont\symletters \ht\usefulbox
    \divide\dimen@ \tw@ 
    \kern\dimen@ 
    \xoverline{\kern-\dimen@ \box\usefulbox\kern\dimen@ }\kern-\dimen@ }}
 {{\setbox\usefulbox=\hbox{$\m@th\scriptstyle #1$}%
    \dimen@ \getslant\the\scriptfont\symletters \ht\usefulbox
    \divide\dimen@ \tw@ 
    \kern\dimen@ 
    \xoverline{\kern-\dimen@ \box\usefulbox\kern\dimen@ }\kern-\dimen@ }}
 {{\setbox\usefulbox=\hbox{$\m@th\scriptscriptstyle #1$}%
    \dimen@ \getslant\the\scriptscriptfont\symletters \ht\usefulbox
    \divide\dimen@ \tw@ 
    \kern\dimen@ 
    \xoverline{\kern-\dimen@ \box\usefulbox\kern\dimen@ }\kern-\dimen@ }}%
 {}}
\makeatother

\makeatletter
\newcommand{\mylabel}[2]{#2\def\@currentlabel{#2}\label{#1}}
\makeatother

\makeatletter
\newcommand{\Mac}{}
\DeclareRobustCommand{\Mac}{%
  M%
  \raisebox{\dimexpr\fontcharht\font`M-\height}{%
    \check@mathfonts\fontsize{\sf@size}{0}\selectfont
    c%
  }%
}
\makeatother

\newtheoremstyle{citing}
  {}
  {}
  {\itshape}
  {}
  {\bfseries}
  {\textbf{.}}
  {.5em}
  {\thmnote{#3}}

\theoremstyle{plain}
\newtheorem{theorem}{Theorem}

\newtheorem{lemma}[theorem]{Lemma}
\newtheorem{corollary}[theorem]{Corollary}

\newtheorem{proposition}[theorem]{Proposition}

\theoremstyle{remark}
\newtheorem{example}[theorem]{Example}

\theoremstyle{definition}

\newtheorem{definition}[theorem]{Definition}

\numberwithin{equation}{section}

\newtheorem*{problem}{Problem}

\theoremstyle{remark}
\newtheorem{remark}[theorem]{Remark}

{\theoremstyle{citing}
}

{\theoremstyle{definition}
}

\title{On the Full Set of Unitarizable Supermodules over $\mathfrak{sl}(m\vert n)$}

\author{Steffen Schmidt}
\address{Center for Quantum Mathematics, 
University of Southern Denmark, DK-5230 Odense, Denmark}
\email{stschmidt@imada.sdu.dk}

\let\origmaketitle\maketitle
\def\maketitle{
  \begingroup
  \def\uppercasenonmath##1{} 
  \let\MakeUppercase\relax 
  \origmaketitle
  \endgroup
}


%% file: main.tex
\section{Introduction} 
\noindent
The classification of unitarizable supermodules of \(\gg\coloneqq\mathfrak{sl}(m\vert n)\) asks which simple supermodules admit a positive definite invariant Hermitian form. Unitarity imposes a strong structural restriction: up to isomorphism, only the real forms $\su(p,q\vert 0,n)$ and $\su(p,q\vert n,0)$ admit nontrivial unitarizable simple supermodules. These real forms are isomorphic, and their unitarizable simple supermodules are of highest and lowest weight type, respectively. By contragredient duality, the classification therefore reduces to determining when the contravariant Hermitian form on a simple highest weight supermodule is positive definite. Unitarity of the highest weight module for the even subalgebra is necessary, but the action of the odd part imposes additional restrictions on the highest weight. 

The purpose of this manuscript is to express these additional restrictions as elementary inequalities on the highest weight. Once the necessary unitarity condition for the even subalgebra is satisfied, we show that unitarity follows from the signs of the pairings $(\Lambda+\rho,\alpha)$ with positive odd roots, with atypicality determining which roots must be considered. For the real forms and positive systems used in the proof, the required sign is positive in the finite-dimensional case and negative in the infinite-dimensional case. The key point is that these individual sign conditions guarantee positivity on the entire representation. This gives a short, uniform proof of the classification in both the compact ($pq=0$) and noncompact ($pq\neq0$) cases, without requiring a complete decomposition under $\even$ or explicit realizations of the supermodules. Odd reflections then express the classification for $\slmn$ in all relevant positive systems.

The starting point is the relative Dirac operator for $(\gg,\even)$, introduced by Huang--Pand\v{z}i\'c \cite{huang2005dirac}, and the unitarity criterion developed in \cite{SchmidtDirac}. We first choose a positive system compatible with the real form $\su(p,q\vert 0,n)$, so that the Dirac operator has the adjointness properties needed for the argument. We use the standard positive system when $pq=0$, corresponding to the finite-dimensional setting, and the nonstandard positive system when $pq\neq0$, corresponding to the infinite-dimensional setting. With these choices, the criterion states that a simple highest weight supermodule $L(\Lambda)$ is unitarizable if and only if its even highest weight module $L_{0}(\Lambda)$ is unitarizable and every $\even$-composition factor $L_{0}(\nu)$ with $\nu\neq\Lambda$ satisfies the strict Dirac inequality (Theorem~\ref{thm::unitarity_via_Dirac_inequality})
\begin{equation}\label{eq::ineqaulity_intro}
(\nu+2\rho,\nu)
\begin{cases}
<(\Lambda+2\rho,\Lambda), & pq=0,\\
>(\Lambda+2\rho,\Lambda), & pq\neq0.
\end{cases}
\end{equation}
Here $\rho$ is the Weyl vector for the chosen positive system. Thus positivity on the entire supermodule is equivalent to unitarity of $L_{0}(\Lambda)$ together with numerical inequalities on the highest weights of the remaining even composition factors. Moreover, to deduce unitarity, it is enough to work with a $\even$-filtration of $L(\Lambda)$; no direct sum decomposition under $\even$ is required. 

The simplest Dirac inequalities correspond to $\even$-composition factors of the form $L_{0}(\Lambda-\alpha)$, where $\alpha$ is a positive odd root. Since $\alpha$ is isotropic, the corresponding necessary conditions for unitarity reduce to
\begin{equation}\label{eq::intro_alpha_inequality}
(\Lambda+\rho,\alpha)
\begin{cases}
>0, & pq=0,\\
<0, & pq\neq0.
\end{cases}
\end{equation}
To organize these inequalities, we arrange the pairings $(\Lambda+\rho,\alpha)$ into matrices reflecting the structure of the even real subalgebra $\su(p,q\vert 0,n)_{\bar0}$. Dominance orders the entries along rows and columns, while their zeros record the atypical roots. This ordering reduces the analysis of the sign conditions in~\eqref{eq::intro_alpha_inequality} to a boundary row or column in each matrix. 

Positivity of the contravariant form imposes further restrictions on these boundary entries and on the highest weight coordinates when an entry vanishes. We call these necessary restrictions the \emph{unitarity conditions} and use them to exclude highest weights that cannot be unitarizable.

For the highest weights satisfying the unitarity conditions, we establish sufficiency as follows. Our results on $\even$-composition factors show that, in these cases, every factor $L_{0}(\nu)$ with $\nu\neq\Lambda$ occurs in an even Verma module $M_{0}(\Lambda-\Gamma_S)$, where $\Gamma_S$ is a sum of distinct positive odd roots satisfying~\eqref{eq::intro_alpha_inequality}. Since the even Casimir acts by the same scalar on a Verma module and its composition factors, we obtain
\begin{equation}
(\nu+2\rho_{\bar0},\nu)
=
(\Lambda-\Gamma_S+2\rho_{\bar0},\Lambda-\Gamma_S).
\end{equation}
This identity allows us to deduce the strict Dirac inequality for $L_{0}(\nu)$. Dominance and integrality ensure that the signs of the individual root pairings give the required inequality for their sum. Replacing $\rho_{\bar0}$ by $\rho$ does not weaken this inequality. The inequalities therefore hold for every composition factor with $\nu\neq\Lambda$, and the Dirac criterion proves unitarity.

The argument uses a positive system compatible with the real form: the standard system in the compact case and the nonstandard system in the noncompact case. Although the change of highest weight under odd reflections is generally difficult to compute, the structure of the unitarizable highest weights allows us to give explicit formulas, including the corrections at atypical weights. These formulas yield the full classification in all relevant positive systems.

Moreover, the classification also yields all simple unitarizable $\mathfrak{psl}(n\vert n)$-supermodules by requiring the central identity of $\sl(n\vert n)$ to act trivially.

\subsection{Comparison to the literature}
The classification has been studied through several approaches. Gould--Zhang \cite{zbMATH04202524} classified finite-dimensional unitarizable supermodules of $\glmn$. For the noncompact case, the principal precedents are the works of Furutsu--Nishiyama, Jakobsen, and G\"unaydin--Volin.

Furutsu--Nishiyama \cite{furutsu1991classification} classified unitarizable supermodules with integral highest weight whose even action integrates to the underlying even Lie group. Their method embeds $\su(p,q\vert0,n)$ into suitable orthosymplectic Lie superalgebras and constructs the representations by restricting oscillator supermodules. This extends the classical oscillator approach to the super setting and provides explicit unitary realizations.

Jakobsen \cite{jakobsen1994full} developed an approach to the full classification, including nonintegral highest weights. His method studies contravariant Hermitian forms on induced modules as the highest weight varies and uses the \emph{principle of the last possible place of unitarity} to determine bounds on the admissible parameters. His intended scope thus includes the integral setting of Furutsu--Nishiyama. However, the argument in the proof of \cite[Theorem 5.2]{jakobsen1994full} does not establish that the $\even$-type used to derive the bounds (5.10)--(5.11) survives after quotienting by the previously generated null submodule. He subsequently notes that such a type may belong to the ideal generated by an earlier vanishing type. The bound is therefore not justified in full generality. G\"unaydin--Volin \cite{Volin} identify omissions in the resulting list, including known doubleton supermultiplets of $\su(2,2\vert N)$.

G\"unaydin--Volin \cite{Volin} give a complete classification, including continuous highest weight parameters. Their result contains the integral cases of Furutsu--Nishiyama and the representations omitted from Jakobsen's list. They first derive necessary inequalities by comparing norms of highest weight vectors for different Borel subalgebras, organizing these inequalities as plaquette constraints. To prove sufficiency, they construct the admissible representations in generalized Fock spaces that allow nonintegral powers of determinant operators. Their classification is the natural point of comparison for ours, and the two agree after translating the highest weight conventions.

Our contribution is a simpler, uniform proof of sufficiency based on individual root signs. It treats finite-dimensional and infinite-dimensional supermodules, including integral and nonintegral highest weights, without requiring explicit realizations or a complete decomposition under $\even$. We also give explicit parametrizations and formulas for the change of highest weight under odd reflections, including atypical corrections. These yield the classification in the standard and nonstandard positive systems without performing the reflections individually. The relative Dirac construction is available for all basic classical Lie superalgebras and provides a framework for extending this approach to their unitarizable supermodules. We will pursue these extensions in forthcoming work.
\vspace{-0.1cm}
\subsection{Organization}
Section~\ref{sec::preliminaries} recalls the structure of $\slmn$, its real forms, and the standard and nonstandard positive systems. We review highest weight supermodules and their Shapovalov forms, construct their $\even$-filtrations, and establish restrictions on even composition factors and their Casimir eigenvalues at atypical weights. We then recall the notion of unitarity and its basic consequences, derive necessary conditions for unitarity, and describe the relevant unitarizable highest weight $\even$-modules.

Section~\ref{sec::Dirac_criterion} introduces the relative Dirac operator and its square formula. We explain the role of positive systems adapted to the anti-involution, derive the strict Dirac inequalities, and give a different proof of the resulting unitarity criterion. We then use this criterion to prove a sufficient condition for unitarity in terms of individual odd root pairings and the possible even composition factors.

Section~\ref{sec::fd} states and proves the finite-dimensional classification and gives explicit parametrizations of the highest weights. 

Section~\ref{sec::ifd} classifies the infinite-dimensional unitarizable simple supermodules and gives an explicit parametrization in the nonstandard positive system. We then compute the change of highest weight under odd reflections to express the classification in the standard positive system. The section concludes with the projective case $\mathfrak{psl}(n\vert n)$.

\subsection{Notation}
We denote by $\ZZ_{+}$ the set of non-negative integers. Let $\ZZ_{2} \coloneqq \ZZ / 2\ZZ$ be the ring of integers modulo $2$. We denote the elements of $\ZZ_2$ by $\overline{0}$ (the residue class of even integers) and $\overline{1}$ (the residue class of odd integers). The ground field is $\CC$, unless otherwise stated.

If $V \coloneqq V_{\bar{0}} \oplus V_{\bar{1}}$ is a super vector space and $v \in V$ is a homogeneous element, then $p(v)$ denotes the parity of $v$, meaning $p(v) = \bar{0}$ if $v \in V_{\bar{0}}$ and $p(v) = \bar{1}$ if $v \in V_{\bar{1}}$.

We fix $\gg\coloneqq\mathfrak{sl}(m\vert n)$ and keep the ordered pair $(m,n)$ fixed. In particular, we do not identify $\mathfrak{sl}(m\vert n)$ with $\mathfrak{sl}(n\vert m)$ by exchanging the even and odd blocks. For the Lie superalgebra $\gg=\even\oplus\odd$, we denote its universal enveloping superalgebra by $\UE(\gg)$. The universal enveloping algebra of the Lie subalgebra $\even$ is denoted by $\UE(\even)$. Their centers are denoted by $\mathfrak{Z}(\gg)$ and $\mathfrak{Z}(\even)$, respectively.

Any $\gg$-supermodule $M$ restricts to a $\even$-supermodule, where each $\even$-module is viewed as concentrated in a single parity. In what follows, this $\ZZ_2$-grading is left implicit, and we simply refer to them as $\even$-modules.
\section{Preliminaries}\label{sec::preliminaries}
\subsection{Structure theory}
\noindent
The \emph{special linear Lie superalgebra} $\slmn$ is defined as the Lie subsuperalgebra of the general linear Lie superalgebra $\glmn$ consisting of all block matrices of the form 
\begin{equation} \label{eq::standard_block_form}
X = \begin{pmatrix}
 A & B \\
 C & D
\end{pmatrix},
\end{equation}
where $A$ is an $m\times m$ complex matrix, $B$ an $m\times n$ complex matrix, 
$C$ an $n\times m$ complex matrix, and $D$ an $n\times n$ complex matrix. 
These matrices are subject to the condition that the supertrace vanishes, \emph{i.e.},
\begin{equation}
\str(X) \coloneqq \tr(A) - \tr(D) = 0. 
\end{equation}
Its natural $\ZZ_{2}$-grading is given by 
$
\slmn = \slmn_{\bar{0}} \oplus \slmn_{\bar{1}},
$
where $\slmn_{\bar{0}}$ consists of all block matrices with $B = C = 0$, and 
$\slmn_{\bar{1}}$ consists of all block matrices with $A = D = 0$ 
(still subject to the vanishing supertrace condition). In what follows, we write $\gg=\slmn$ and assume $m+n>2$, thus excluding the nilpotent case $m=n=1$. Since $\sl(m\vert n)\cong\sl(n\vert m)$, we may assume $m\leq n$. Having made this choice, we keep the ordered pair $(m,n)$ fixed throughout.

We briefly review the structure theory of $\gg \coloneqq \slmn$, with a focus on the real forms
$\su(p,q\vert r,s)$. 

\subsubsection{Root data}
Let $\hh\subset\gg$ be the subalgebra of diagonal matrices with vanishing supertrace. 
For $1\leq i\leq m$ and $1\leq k\leq n$, define
$\epsilon_i,\delta_k\in\hh^{*}$ by
\begin{equation}
 \epsilon_i(H)=h_i,
 \qquad
 \delta_k(H)=h_{m+k},
\end{equation}
for $H=\operatorname{diag}(h_1,\ldots,h_{m+n})\in\hh.$ These functionals span $\hh^{*}$ and satisfy the relation
\begin{equation}
\sum_{i=1}^{m}\epsilon_i-\sum_{k=1}^{n}\delta_k=0.
\end{equation}
Accordingly, every weight $\lambda\in\hh^{*}$ can be written as
\begin{equation}\label{eq::standard_coordinates}
 \lambda
 =
 \lambda^{1}\epsilon_{1}+\cdots+\lambda^{m}\epsilon_{m}
 +\mu^{1}\delta_{1}+\cdots+\mu^{n}\delta_{n} = (\lambda^{1},\ldots,\lambda^{m}\vert\mu^{1},\ldots,\mu^{n}).
\end{equation}
This representation is not unique: shifting the tuple by a multiple of
$
(1,\ldots,1\vert-1,\ldots,-1)
$
does not change the corresponding weight. Since $\hh \subset \even$ is a Cartan subalgebra for $\even$, its action on any finite-dimensional simple $\even$-module is diagonalizable. Hence the adjoint action of $\hh$ on $\gg$ is diagonalizable, and $\gg$ admits a root space decomposition:
\begin{equation}
\gg = \hh \oplus \bigoplus_{\alpha \in \hh^{\ast}\setminus\{0\}} \gg^{\alpha}, 
\qquad \gg^{\alpha}\coloneqq \{ X \in \gg \ : \ [H,X] = \alpha(H)X \ \text{for all} \ H \in \hh\}.
\end{equation}
The \emph{set of roots} is $\Delta=\Delta_{\bar{0}}\sqcup \Delta_{\bar{1}}$, where
\begin{equation}
\label{allroots}
\begin{split}
\Delta_{\bar{0}} &= \{\pm(\epsilon_{i}-\epsilon_{j}),\pm(\delta_{a}-\delta_{b}) \ : \ 1\leq i<j\leq m, \ 1\leq a <b \leq n\}, \\
\Delta_{\bar{1}} &= \{\pm(\epsilon_{i}-\delta_{a}) \ : \ 1\leq i \leq m, \ 1\leq a\leq n\},
\end{split}
\end{equation}
are the \emph{even} and \emph{odd roots}, respectively. 

Except when \(\mathfrak g=\mathfrak{sl}(2\vert 2)\), each root space has superdimension either \((1\vert 0)\) or \((0\vert 1)\). For \(\mathfrak{sl}(2\vert 2)\), the even root spaces have superdimension \((1\vert 0)\), whereas the odd root spaces have superdimension \((0\vert 2)\). The set $\Delta_{\bar{0}}$ decomposes as the disjoint union of the root systems 
of $\sl(m)$ and $\sl(n)$. 
For the even part, we choose once and for all the standard system of positive roots
\begin{equation}
\label{evenpositive}
\Delta_{\bar{0}}^{+} \coloneqq \{\, \epsilon_{i}-\epsilon_{j}, \ \delta_{a}-\delta_{b} 
 \ : \ 1 \leq i < j \leq m, \ 1 \leq a < b \leq n \,\},
\end{equation}
so that the root vectors associated with $\epsilon_{i}-\epsilon_{j}$ for $i<j$ 
are realized as strictly upper triangular matrices in $\sl(m)$, while those associated 
with $\delta_{k}-\delta_{l}$ for $k<l$ are strictly upper triangular matrices in $\sl(n)$. Here, $\sl(m)$ and $\sl(n)$ are embedded as diagonal blocks in $\even$. 

Fix $\Delta_{\bar{1}}^{+}\subset \Delta_{\bar{1}}$ such that
$
\Delta^{+}
\coloneqq
\Delta_{\bar{0}}^{+}\sqcup\Delta_{\bar{1}}^{+}
$
is a positive system. If $\gg\neq\mathfrak{sl}(2|2)$, then $\gg$ admits the \emph{triangular decomposition}
\begin{equation}
\gg = \nn^{-}\oplus \hh \oplus \nn^{+},
\qquad
\nn^{\pm} \coloneqq \bigoplus_{\alpha \in \Delta^{+}} \gg^{\pm \alpha},
\end{equation}
where $\nn^{\pm}$ are nilpotent Lie subsuperalgebras, and the even part reproduces the 
usual triangular decomposition of $\even$. For $\gg=\mathfrak{sl}(2\vert 2)$, the positive roots and root spaces in the definition of $\nn^\pm$ are taken relative to the full diagonal algebra of $\mathfrak{gl}(2\vert 2)$, while $\hh$ remains the diagonal Cartan subalgebra of $\gg$.

The corresponding \emph{Borel subalgebra} is $\bb = \hh \oplus \nn^{+}$. Moreover, we define the corresponding \emph{Weyl vector} by
$
\rho
\coloneqq
\rho_{\bar{0}}-\rho_{\bar{1}},
$
where
\begin{equation}\label{eq::Weyl_vector}
\rho_{\bar{0}}
\coloneqq
\tfrac{1}{2}\sum_{\alpha\in\Delta_{\bar{0}}^{+}}\alpha,
\qquad
\rho_{\bar{1}}
\coloneqq
\tfrac{1}{2}\sum_{\alpha\in\Delta_{\bar{1}}^{+}}\alpha.
\end{equation}
Here, note that the definition of $\rho$ only depends on the choice of $\Delta_{\bar{1}}^{+}$ since we fixed $\Delta_{\bar{0}}^{+}$. 

The \emph{simple system} $\Uppi \subset \Delta^{+}$ associated to $\Delta^{+}$ is the set of all $\alpha \in \Delta^{+}$ which cannot be written as the sum of two roots in $\Delta^{+}$. Elements of $\Uppi$ are called \emph{simple} roots. For $\Uppi = \{\alpha_{1},\dots,\alpha_{r}\}$, any $\alpha \in \Delta$ can be represented as a linear combination
$
\alpha = \sum_{i = 1}^{r} k_{i}\alpha_{i},
$
where either all $k_{i} \in \ZZ_{\geq 0}$ or all $k_{i} \in \ZZ_{\leq 0}$. If $m\neq n$, this is unique. 

A simple system inducing $\Delta_{\bar0}^{+}$ may be encoded by an ordered weight word
$s=(s_1,\ldots,s_{m+n})$ in which each element of
$\{\epsilon_1,\ldots,\epsilon_m,\delta_1,\ldots,\delta_n\}$
occurs exactly once, while the relative orders of the $\epsilon_i$ and of the $\delta_a$ are preserved. The corresponding simple system is
$\Uppi_s\coloneqq\{s_i-s_{i+1}:1\leq i<m+n\}$.

Simple systems and positive systems are in bijection. Two positive systems/simple systems are relevant to the study of unitarizable
$\gg$-supermodules (see Section~\ref{subsubsec::unitarizable_even_HWM}).

\begin{example}\label{ex::systems} Let $m=p+q$ with $p,q\geq 0$. \begin{enumerate} \item[(a)] The standard simple system and positive system are \begin{equation*}\begin{aligned}\Uppi_{\mathrm{st}} &=\{\epsilon_i-\epsilon_{i+1}:1\leq i<m\} \cup\{\epsilon_m-\delta_1\} \cup\{\delta_a-\delta_{a+1}:1\leq a<n\}, \\
s_{\mathrm{st}}
&=
(\epsilon_1,\ldots,\epsilon_m,\delta_1,\ldots,\delta_n),\\\Delta_{\mathrm{st}}^+ &=\Delta_{\bar{0}}^+\sqcup \{\epsilon_i-\delta_a:1\leq i\leq m,\ 1\leq a\leq n\}, \\ \rho_{\st} &= \tfrac{1}{2}\sum_{i=1}^{m}(m-n-2i+1)\epsilon_i
+
\tfrac{1}{2}\sum_{j=1}^{n}(m+n-2j+1)\delta_j.
\end{aligned}\end{equation*}
In this case, the odd positive root vectors span $\gg_{+1}$. We denote the associated Borel subalgebra by $\bb_{\text{st}}$.

\item[(b)] The $pq$-non-standard simple system and positive system are
\begin{equation*}
\begin{aligned}
\Uppi_{\mathrm{nst}}
&=
\{\epsilon_i-\epsilon_{i+1}:1\leq i<p\}
\cup\{\epsilon_p-\delta_1\}
\cup\{\delta_a-\delta_{a+1}:1\leq a<n\} \\
&\qquad\cup\{\delta_n-\epsilon_{p+1}\}
\cup\{\epsilon_j-\epsilon_{j+1}:p<j<m\},\\
s_{\mathrm{nst}}
&=
(\epsilon_1,\ldots,\epsilon_p,\delta_1,\ldots,\delta_n,
\epsilon_{p+1},\ldots,\epsilon_m),\\
\Delta_{\mathrm{nst}}^+
&=
\Delta_{\bar0}^+
\sqcup
\{\epsilon_i-\delta_a:1\leq i\leq p,\ 1\leq a\leq n\}
\sqcup
\{\delta_a-\epsilon_j:p<j\leq m,\ 1\leq a\leq n\}, \\
\rho_{\nst}&=
\tfrac{1}{2}\sum_{i=1}^{p}(m-n-2i+1)\epsilon_i
+
\tfrac{1}{2}\sum_{i=p+1}^{m}(m+n-2i+1)\epsilon_i
+
\tfrac{1}{2}\sum_{k=1}^{n}(n+p-q-2k+1)\delta_k.
\end{aligned}
\end{equation*}
Here, \(\{\epsilon_p-\delta_1\}\) is omitted when \(p=0\), and \(\{\delta_n-\epsilon_{p+1}\}\) is omitted when \(q=0\). We denote the associated Borel subalgebra by $\bb_{\mathrm{nst}}$. For $q=0$, it reduces to the standard positive system, while for $p=0$ the positive even roots remain unchanged and all odd roots have the opposite sign.
 \end{enumerate} \end{example}

Next, we equip $\gg$ with the supersymmetric invariant bilinear form induced by the supertrace. The supertrace defines an even, supersymmetric, invariant bilinear form
\begin{equation}
(X,Y)\coloneqq\operatorname{str}(XY),
\qquad X,Y\in\gg.
\end{equation}
Thus $(\cdot,\cdot)$ is symmetric on $\gg_{\bar0}$, skew-symmetric on $\gg_{\bar1}$, and $\gg_{\bar0}$ is orthogonal to $\gg_{\bar1}$. Its invariance is expressed by $([X,Y],Z)=(X,[Y,Z])$ for all $X,Y,Z\in \gg$.
If $m\neq n$, this form is non-degenerate. If $m=n$, its radical is the one-dimensional center $\mathbb C E_{2n}$.

Let $\dd$ denote the diagonal subalgebra of $\mathfrak{gl}(m\vert n)$, so that $\hh=\dd\cap\gg$. The supertrace form is non-degenerate on $\dd$ and hence identifies $\dd$ with $\dd^*$. With respect to the coordinate functionals $\epsilon_i,\delta_a\in\dd^*$, one has
\begin{equation}
(\epsilon_i,\epsilon_j)=\delta_{ij},
\qquad
(\delta_a,\delta_b)=-\delta_{ab},
\qquad
(\epsilon_i,\delta_a)=0.
\end{equation}
Every root of $(\gg,\hh)$ is the restriction to $\hh$ of a root on $\dd$. Viewing $\alpha\in\Delta$ in this way, let $h_\alpha\in\dd$ be determined by $\alpha(H)=(H,h_\alpha)$ for all $H\in\dd.
$ In fact, $h_\alpha\in\hh$. We define
\begin{equation}
(\alpha,\beta)\coloneqq(h_\alpha,h_\beta),
\qquad
\alpha,\beta\in\Delta.
\end{equation}
For $m\neq n$, this agrees with the bilinear form on $\hh^*$ induced by the non-degenerate restriction of $(\cdot,\cdot)$ to $\hh$. When $m=n$, the restriction to $\hh$ is degenerate, but the above construction still defines the standard bilinear form on the root lattice. In particular, every odd root is isotropic.

The root system $\Delta$ admits an action of the Weyl group $W$ of the even part $\even$. 
The group $W$ is isomorphic to $S_{m}\times S_{n}$, the product of the symmetric groups on $m$ and $n$ letters. 
It is generated by the reflections with respect to the even roots,
\begin{equation}
\label{evenreflection}
r_{\alpha}(\beta)=\beta-2\frac{(\alpha,\beta)}{(\alpha,\alpha)}\,\alpha,
\qquad \alpha\in\Delta_{\bar{0}},\ \beta\in\Delta.
\end{equation}
The Weyl group action preserves the bilinear form on the root lattice and if $m\neq n$, its extension to $\hh^*$ preserves the induced bilinear form there. 
The Weyl group of $\gg$ is by definition that of $\even$. 
The dot action of $W$ on $\hh^{*}$ is defined by
\begin{equation}
w\cdot\lambda=w(\lambda+\rho)-\rho,
\qquad \lambda\in\hh^{*},\ w\in W.
\end{equation}
Two weights $\lambda,\nu\in\hh^{*}$ are $W$-linked if $\nu=w\cdot\lambda$ for some $w\in W$. This defines an equivalence relation on $\hh^{*}$. 
The equivalence class $\{w\cdot\lambda:\,w\in W\}$ is called the $W$-linkage class of $\lambda$.

The simple systems $\Uppi_{\st}$ and $\Uppi_{\nst}$ are not conjugate under $W$ if $pq\neq 0$. To relate them, we use a sequence of \emph{odd reflections}. Given an odd simple root $\theta \in \Uppi$, define the \emph{odd reflection} by
\begin{equation}\label{oddreflection}
r_{\theta}(\alpha) =
\begin{cases}
\alpha + \theta & \text{if } \alpha \neq \theta\ \text{and} \ (\alpha,\theta) \neq 0, \\
\alpha & \text{if } \alpha \neq \theta \ \text{and} \ (\alpha,\theta) = 0, \\
-\theta & \text{if } \alpha = \theta,
\end{cases}
\end{equation}
for any $\alpha \in \Uppi$. Then $\Delta_{\theta}^{+}=\{-\theta\}\cup(\Delta^{+}\setminus\{\theta\})$ is a new positive system with simple system $\Uppi_{\theta}\coloneqq r_{\theta}(\Uppi)$. If $\Uppi$ and $\Uppi'$ are two simple systems such that $\Delta^{+}_{\bar{0}}=(\Delta')^{+}_{\bar{0}}$, then $\Uppi'$ can be obtained from $\Uppi$ by a sequence of odd reflections \cite[Proposition 1]{Serganova_reps}. For $\Uppi_{\st}$ and $\Uppi_{\nst}$, this sequence is given by the following lemma.
\begin{lemma}[{\cite{Schmidt_unitary_branching}}]\label{lemm::relating_standard_and_non_standard_simple_system}
For $m=p+q$ with $p,q\geq 0$, one has
\begin{equation*}
\Uppi_{\mathrm{st}}
=
\bigl(r_{\delta_1-\epsilon_m}\cdots r_{\delta_n-\epsilon_m}\bigr)
\cdots
\bigl(r_{\delta_1-\epsilon_{p+1}}\cdots r_{\delta_n-\epsilon_{p+1}}\bigr)
\Uppi_{\mathrm{nst}}.
\end{equation*}
\end{lemma}

\subsubsection{Real forms \texorpdfstring{$\mathfrak{su}(p,q\vert r,s)$}{}}
A real form of $\gg$ is a real Lie superalgebra $\gg_{\mathbb R}$ such that $\gg_{\mathbb R}\otimes_{\mathbb R}\CC\cong\gg$, equivalently, the fixed-point algebra of a parity-preserving conjugate-linear involution $\sigma$ of $\gg$. We call such an involution consistent with the chosen root decomposition if \(\sigma(\gg^\alpha)=\gg^{-\alpha}\) for all \(\alpha\in\Delta\). To define unitarity, we use a conjugate-linear anti-involution $\omega$ satisfying $\omega^2=\operatorname{id}$ and $\omega([X,Y])=[\omega(Y),\omega(X)]$ for homogeneous $X,Y\in\gg$ \cite{SchmidtDirac,Schmidt,SchmidtWalcher}. The two descriptions are related by
\begin{equation}
\sigma_\omega(X)=-(-i)^{p(X)}\omega(X),\qquad
\gg_{\mathbb R}=\bigoplus_{\bar{a}\in \ZZ_{2}}\{X\in\gg_{\bar a}:\omega(X)=-i^{\bar{a}}X\}.
\label{eq:real-form-anti-involution}
\end{equation}
The real forms of $\sl(m\vert n)$ were classified by Parker \cite{ParkerRealForms} and Serganova \cite{serganova1983classification}. They are $\mathfrak{sl}(m\vert n;\mathbb R)$, $\mathfrak{su}(p,q\vert r,s)$ with $p+q=m$ and $r+s=n$, and, when both $m$ and $n$ are even, $\mathfrak{su}^{\ast}(m\vert n)$; for $m=n$, the additional real forms $\mathfrak q(n)$ and $\mathfrak{up}(n)$ occur. Among these real forms, only the real forms $\su(p,q\vert r,s)$ are of interest (see Theorem~\ref{thm::HW_property_M}) for this article.

We realize the real forms $\su(p,q\vert r,s)$ as the $\ZZ_{2}$-graded fixed-point sets associated with the conjugate-linear anti-involution
\begin{equation}
\omega(X)=J_{(p,q\vert r,s)}^{-1}X^\dagger J_{(p,q\vert r,s)},\qquad X\in\gg,\qquad
J_{(p,q\vert r,s)}\coloneqq
\left(\begin{array}{@{}c|c@{}}
I_{p,q} & 0 \\
\hline
0 & I_{r,s}
\end{array}\right),
\end{equation}
where $I_{k,l}$ is diagonal with $k$ entries equal to $1$ followed by $l$ entries equal to $-1$, and $(\cdot)^\dagger$ denotes conjugate transpose. Among these real forms, only the isomorphic forms $\su(p,q\vert n,0)$ and $\su(p,q\vert 0,n)$ admit nontrivial unitarizable supermodules (Theorem~\ref{thm::HW_property_M}). The corresponding anti-involutions have two forms, according as $p=0$ or $q=0$, or $p,q\neq0$. Any $X\in\sl(m\vert n)$ has the form \footnote{Let
$
\kk\coloneqq \mathfrak{s}(\mathfrak{u}(p)\oplus\mathfrak{u}(q))\oplus\su(n)\oplus\mathfrak{u}(1)
$
be a maximal compact subalgebra of the even part of $\su(p,q\vert 0,n)\cong\su(p,q\vert n,0)$. Then the block form reflects the decomposition of $\sl(m\vert n)$ under $\kk^{\CC}$.}
\begin{equation}\label{eq::general_form}
X=\left(\begin{array}{@{}c|c@{}}
\begin{matrix}
a & b \\ c & d
\end{matrix}
&
\begin{matrix}
P_{1} \\ P_{2}
\end{matrix} \\
\hline
\begin{matrix}
Q_{1} & Q_{2}
\end{matrix}
&
E
\end{array}\right),
\end{equation}
where $P_{1}$ is a $p\times n$ matrix, $P_{2}$ a $q\times n$ matrix, $Q_{1}$ an $n\times p$ matrix, $Q_{2}$ an $n\times q$ matrix,
$
\left(\begin{smallmatrix} a & b \\ c & d \end{smallmatrix}\right)\in\mathfrak{gl}(m,\CC),
$
$E\in\mathfrak{gl}(n,\CC)$, and
$
\tr\left(\begin{smallmatrix} a & b \\ c & d \end{smallmatrix}\right)-\tr E=0.
$ If $p=0$ or $q=0$, we use the standard block form~\eqref{eq::standard_block_form}.

\begin{lemma}[{\cite[Lemma 4.1]{jakobsen1994full}, \cite{SchmidtDirac}}]\label{lemm::anti_involutions}
Up to conjugation by automorphisms, there are exactly two consistent conjugate-linear anti-involutions on $\gg$ compatible with the standard ordering and inducing the even real forms $\su(p,q\vert 0,n)_{\bar{0}}\cong\su(p,q\vert n,0)_{\bar{0}}$:
\begin{enumerate}
\item[a)] If $p=0$ or $q=0$, then
\[
\omega_{\pm}
\left(
\begin{array}{@{}c|c@{}}
A & B \\
\hline
C & D
\end{array}
\right)
=
\left(
\begin{array}{@{}c|c@{}}
A^{\dagger} & \pm C^{\dagger} \\
\hline
\pm B^{\dagger} & D^{\dagger}
\end{array}
\right).
\]
The real Lie superalgebra $\su(m,0\vert n,0)\cong\su(0,m\vert 0,n)$ corresponds to $\omega_{+}$, whereas $\su(m,0\vert 0,n)\cong\su(0,m\vert n,0)$ corresponds to $\omega_{-}$.

\item[b)] If $p,q\neq0$, then
\begin{equation*}
\omega_{(-,+)}(X)=
\left(\begin{array}{@{}c|c@{}}
\begin{matrix}
a^{\dagger} & -c^{\dagger} \\ -b^{\dagger} & d^{\dagger}
\end{matrix}
&
\begin{matrix}
-Q_{1}^{\dagger} \\ Q_{2}^{\dagger}
\end{matrix} \\
\hline
\begin{matrix}
-P_{1}^{\dagger} & P_{2}^{\dagger}
\end{matrix}
&
\begin{matrix}
E^{\dagger}
\end{matrix}
\end{array}\right),\qquad
\omega_{(+,-)}(X)=
\left(\begin{array}{@{}c|c@{}}
\begin{matrix}
a^{\dagger} & -c^{\dagger} \\ -b^{\dagger} & d^{\dagger}
\end{matrix}
&
\begin{matrix}
Q_{1}^{\dagger} \\ -Q_{2}^{\dagger}
\end{matrix} \\
\hline
\begin{matrix}
P_{1}^{\dagger} & -P_{2}^{\dagger}
\end{matrix}
&
\begin{matrix}
E^{\dagger}
\end{matrix}
\end{array}\right).
\end{equation*}
The real Lie superalgebras $\su(p,q\vert 0,n)$ and $\su(p,q\vert n,0)$ correspond to $\omega_{(-,+)}$ and $\omega_{(+,-)}$, respectively.
\end{enumerate}
\end{lemma}

\begin{remark}
\begin{enumerate}
\item Note that $\omega_{+}(X)=(-1)^{p(X)}\omega_{-}(X)$ and $\omega_{(-,+)}(X)=(-1)^{p(X)}\omega_{(+,-)}(X)$ for all homogeneous $X\in\gg$.
\item If $pq=0$, then $\su(p,q\vert 0,n)_{\bar{0}}\cong\su(p,q\vert n,0)_{\bar{0}}\cong\su(m)\oplus\su(n)\oplus\u(1)$, and the even part is compact. If $pq\neq0$, then $\su(p,q\vert 0,n)_{\bar{0}}\cong\su(p,q\vert n,0)_{\bar{0}}\cong\su(p,q)\oplus\su(n)\oplus\u(1)$, and the even part is non-compact. Thus $pq=0$ and $pq\neq0$ distinguish the compact and non-compact cases, respectively.
\end{enumerate}
\end{remark}

\subsection{Highest weight supermodules}
Fix a triangular decomposition $\gg=\nn^{-}\oplus\hh\oplus\nn^{+}$ and let $\bb \coloneqq \hh \oplus \nn^{+}$ be the associated Borel subalgebra.
A $\gg$-supermodule $M$ is called \emph{highest} (resp.\ \emph{lowest) weight} if there exist
$\Lambda\in\hh^{*}$ and a nonzero vector $v_{\Lambda}\in M$ such that:
\begin{enumerate}
 \item[a)] $Xv_{\Lambda}=0$ for all $X \in \nn^{+}$ (resp. $X \in \nn^{-}$), 
 \item[b)] $Hv_{\Lambda}=\Lambda(H)v_{\Lambda}$ for all $H \in \hh$, and 
 \item[c)] $\mathfrak{U}(\gg)v_{\Lambda}=M$.
\end{enumerate} 
The vector $v_{\Lambda}$ is called a \emph{highest} (resp.\ \emph{lowest}) weight vector of $M$. 
We call $M$ \emph{even} or \emph{odd} according to the parity of $v_\Lambda$. In this article, we place $v_{\Lambda}$ in even degree. Accordingly, all our results are understood up to application of the parity reversion functor $\Pi$.

Every highest weight $\gg$-supermodule $M$ is a weight supermodule, that is, the action of $\hh$ on $M$ is semisimple. Hence
$
M=\bigoplus_{\mu\in\hh^{\ast}}M^{\mu},
$
where
$
M^{\mu}=\{v\in M\mid Hv=\mu(H)v\text{ for all }H\in\hh\}
$
is the $\mu$-weight space of $M$. 

Simple highest weight supermodules are the unique simple quotients of Verma supermodules, which we shortly recall. The general theory can be found in \cite[Section~8]{musson2012lie}.
The enveloping superalgebra $\UE(\gg)$ is viewed as a right $\UE(\bb)$-supermodule by right multiplication. 
For $\Lambda \in \hh^{\ast}$, let $\CC_{\Lambda}$ denote the one-dimensional $\UE(\bb)$-supermodule of weight $\Lambda$, 
with trivial action of $\nn^{+}$. 
The \emph{Verma supermodule} of highest weight $\Lambda$ is then defined by
\begin{equation}\label{eq::Verma_supermodules}
M^{\bb}(\Lambda) = \UE(\gg) \otimes_{\UE(\bb)} \CC_{\Lambda}.
\end{equation}

The Verma supermodule $M^{\bb}(\Lambda)$ is a highest weight $\gg$-supermodule generated by the vector $[1\otimes 1]$ of weight $\Lambda$, with $M^{\bb}(\Lambda) = \UE(\nn^{-})[1\otimes 1]$. For any $\gg$-supermodule $M$ and any $\bb$-eigenvector $v_{\Lambda}\in M$ of weight $\Lambda$, there exists a unique morphism $M^{\bb}(\Lambda)\to M$ sending $[1\otimes1]$ to $v_\Lambda$, whose image is $\UE(\gg)v_\Lambda$. In particular, it is surjective if $v_\Lambda$ generates $M$. The module $M^{\bb}(\Lambda)$ admits a unique maximal subsupermodule, hence a unique simple quotient $L^{\bb}(\Lambda)$. If $\Lambda$ is the highest weight of a unitarizable $\even$-module, $L^{\bb}(\Lambda)$ is an example of a Harish-Chandra supermodule (\emph{cf.}~\cite{Carmeli_Fioresi_Varadarajan_HW}), that is, $L^{\bb}(\Lambda)$ is finitely generated and locally finite as a $\kk^{\CC}$-supermodule. If $\bb$ is fixed, or is clear from the context, we omit the superscript in $M^{\bb}(\Lambda)$ and $L^{\bb}(\Lambda)$.

The definition of the Verma supermodule depends on the choice of a positive system. In Example~\ref{ex::systems}, we introduced two positive systems which will be relevant for the study of unitarity, namely the standard one $\Delta^{+}_{\st}$ and the non-standard one $\Delta^{+}_{\nst}$, which share the same even part $\Delta^{+}_{\bar{0}}$ fixed in~\eqref{evenpositive}. The corresponding Borel subalgebras are denoted $\bb_{\st}$ and $\bb_{\nst}$. The two systems are connected by a sequence of odd reflections defined in~\eqref{oddreflection}. Under these reflections one has \cite[Lemma~1.40]{cheng2012dualities}
\begin{equation} \label{eq::relation_Verma_modules_and_systesm}
 L^{\bb_{\st}}(\Lambda) \;\cong\; \Pi^{\epsilon}L^{\bb_{\nst}}(\Lambda'),
\end{equation}
for some $\epsilon\in \ZZ_{2}$, where the highest weights $\Lambda,\Lambda' \in \hh^{\ast}$ are related by the odd reflection operations. The superscript indicates the choice of Borel subalgebra. Moreover, $\Lambda'$ is typical if and only if $\Lambda$ is typical (\emph{cf.}~\cite[Cor.~9.3.5]{musson2012lie}).

\begin{remark}\label{rmk::useful_identity_odd_reflection}
The following identity is useful. Fix a simple system $\Uppi$, let $\theta\in\Uppi$ be odd, and set $\Uppi'\coloneqq r_{\theta}\Uppi$. Denote the corresponding Weyl vectors by $\rho$ and $\rho'$, and let $\lambda'$ be the highest weight with respect to $\Uppi'$. Then
\[
\lambda'+\rho'=
\begin{cases}
\lambda+\rho,&(\lambda,\theta)\neq0,\\
\lambda+\rho+\theta,&(\lambda,\theta)=0.
\end{cases}
\]
\end{remark}

We fix $\Delta^{+}_{\st}$ when $p=0$ or $q=0$, and $\Delta^{+}_{\nst}$ when $p,q\neq 0$. The standard positive system $\Delta^{+}_{\st}$ is convenient because it induces a natural $\ZZ$-grading on $\gg$ compatible with its $\ZZ_{2}$-grading. The nonstandard positive system $\Delta^{+}_{\nst}$, on the other hand, is better adapted to unitarity and to infinite-dimensional highest weight supermodules, and hence to the real forms $\su(p,q\vert 0,n)\cong \su(p,q\vert n,0)$ with $pq\neq 0$. We therefore work with $\Delta^{+}_{\nst}$ in this case. The corresponding classification for $\Delta^{+}_{\st}$ is then obtained by a sequence of odd reflections.

\subsubsection{Shapovalov form} \label{subsec::Shapovalov} Let $\omega$ be a conjugate-linear anti-involution of $\gg$ that is consistent, that is, $\omega(\gg^{\alpha})=\gg^{-\alpha}$ for all $\alpha \in \Delta$. A weight $\Lambda\in\hh^{\ast}$ is called \emph{symmetric} if
$\Lambda(\omega(H))=\overline{\Lambda(H)}$
for every $H\in\hh$. For such a weight, the Verma supermodule $M^{\bb}(\Lambda)$ carries a natural $\omega$-contravariant Hermitian form, which descends to its unique simple quotient $L^{\bb}(\Lambda)$. This is the \emph{Shapovalov form}. The highest weight of every unitarizable highest weight $\gg$-supermodule is symmetric, and $L^{\bb}(\Lambda)$ is unitarizable precisely when its Shapovalov form is positive definite. We now fix a positive system, omit the superscript $\bb$, and briefly recall the construction.

By the PBW theorem, there is a decomposition of super vector spaces
\begin{equation}
\UE(\gg)=\UE(\hh)\oplus\bigl(\nn^{-}\UE(\gg)+\UE(\gg)\nn^{+}\bigr).
\end{equation}
This decomposition is preserved by the extension of $\omega$ to $\UE(\gg)$. The corresponding projection $p:\UE(\gg)\to\UE(\hh)$ is the \emph{Harish-Chandra projection}. Now assume that $\Lambda \in \hh^{\ast}$ is symmetric. Extend it uniquely to an algebra homomorphism $\Lambda: \UE(\hh)\to \CC$ and set
\begin{equation} \label{eq::definition_F}
F(X,Y) \coloneqq \Lambda(\operatorname{p}(\omega(X)Y)),\qquad X,Y\in\UE(\gg).
\end{equation}
Recall $M(\Lambda) = \UE(\nn^{-})[1\otimes 1]$. Then
\begin{equation}
 \langle X[1\otimes 1],\,Y[1\otimes 1]\rangle\coloneqq F(X,Y),\qquad X,Y\in\UE(\nn^{-}),
\end{equation}
defines an $\omega$-contravariant Hermitian form on $M(\Lambda)$. $\omega$-contravariance means
$
 \langle Xv,w\rangle=\langle v,\omega(X)w\rangle$ for $X\in\gg$ and $v,w\in M(\Lambda).
$ This form is called the \emph{Shapovalov form} on $M(\Lambda)$.
The form is unique up to nonzero scalar multiples and satisfies
\begin{equation} \label{eq::orthogonal_weight_spaces}
 \langle M(\Lambda)^{\mu}, M(\Lambda)^{\nu}\rangle = 0 
 \quad \text{unless} \quad \mu = \nu.
\end{equation}
The Shapovalov form induces on $L(\Lambda)$ a non-degenerate $\omega$-contravariant form, denoted again by $\langle\cdot,\cdot\rangle$. We call the induced form $\langle\cdot,\cdot\rangle$ on $L(\Lambda)$ again the \emph{Shapovalov form}. 

\subsubsection{ \texorpdfstring{$\even$}{}-filtration}\label{subsubsec::even_filtration} We fix a positive system $\Delta^{+}$ and denote the Verma module simply by $M(\Lambda)$ for $\Lambda \in \hh^{\ast}$. The associated Verma module of $\even$ with respect to $\Delta_{\bar{0}}^{+}$ is denoted by $M_{0}(\Lambda)$. We denote the simple quotients by $L(\Lambda)$ and $L_{0}(\Lambda)$, respectively. 
We are concerned with the structure of $L(\Lambda)$ as a $\even$-module, neglecting parity. 
To this end, we recall a Verma $\even$-filtration due to Musson~\cite[Chap.~10]{musson2012lie}, deduce a $\even$-filtration of $L(\Lambda)$ and use a result of Jakobsen to describe the explicit form of the highest weights. 

Let $\Gamma$ denote the set of sums of distinct odd positive roots. For each $\gamma \in \Gamma$, let $p(\gamma)$ be the number of distinct partitions of $\gamma$ into odd positive roots. Fix an ordering of $\Delta_{\bar1}^+$. For $S\subseteq\Delta_{\bar1}^+$, set
$
\Gamma_S \coloneqq \sum_{\alpha \in S}\alpha
$
and denote the corresponding ordered product by
$
X_{-S} \coloneqq \prod_{\alpha \in S} X_{-\alpha},
$
where $X_{-\alpha}$ is the root vector corresponding to $-\alpha \in \Delta_{\bar{1}}^{-}$, with $\Delta_{\bar{1}}^{-} \coloneqq -\Delta_{\bar{1}}^{+}$. Choose an ordering $S_{1},\dots,S_{N}$ of the subsets of $\Delta_{\bar{1}}^{+}$ such that $\Gamma_{S_i} < \Gamma_{S_j}$ implies $i<j$. Define
\begin{equation} \label{eq::construction_Jorad_Hölder}
M_{k} \coloneqq \bigoplus_{1 \leq j \leq k} \UE(\nn_{\bar{0}}^{-}) \, X_{-S_{j}}[1\otimes 1].
\end{equation}
By the PBW theorem this sum is direct. A direct computation shows that the image of $X_{-S_{i+1}}[1\otimes 1]$ in $M(\Lambda)/M_{i}$ is a highest weight vector, and the submodule it generates is isomorphic to $M_{0}(\Lambda - \Gamma_{S_{i+1}})$.

\begin{proposition}[{\cite[Theorem 10.4.5]{musson2012lie}}] \label{prop::filtration_Verma_supermodule}
The supermodule $M(\Lambda)$ has a filtration as a $\even$-module:
$$
0 = M_{0} \subset M_{1} \subset M_{2} \subset \ldots \subset M_{N} = M(\Lambda),
$$
such that each factor $M_{i+1}/M_{i}$ is isomorphic to a Verma module $M_{0}(\Lambda - \gamma)$, where $\gamma \in \Gamma$. This module appears with multiplicity $p(\gamma)$ in the filtration.
\end{proposition}

This yields a $\even$-filtration of $L(\Lambda)$.

\begin{corollary}\label{cor::filtration_simple_supermodule}
Let $0=M_0\subset M_1\subset\cdots\subset M_r=M(\Lambda)$ be the $\gg_{\bar 0}$-filtration from Proposition~\ref{prop::filtration_Verma_supermodule}, and let $\pi:M(\Lambda)\longrightarrow L(\Lambda)$
be the canonical quotient map. Then $L(\Lambda)$ has a finite $\even$-filtration
$
0=\pi(M_0)\subseteq \pi(M_1)\subseteq\cdots\subseteq
\pi(M_r)=L(\Lambda)
$ and $\pi(M_{i+1})/\pi(M_i)$
is either zero or a quotient of $M_0(\Lambda-\gamma)$.
\end{corollary}
\begin{proof}
Since $\pi$ is a $\gg_{\bar 0}$-module homomorphism, the subspaces
$\pi(M_i)$ are $\gg_{\bar 0}$-submodules of
$L(\Lambda)$ and form an increasing filtration. For each $i$, the restriction of $\pi$ to $M_{i+1}$ induces a surjective
$\gg_{\bar 0}$-module homomorphism
\[
M_{i+1}/M_i
\longrightarrow
\pi(M_{i+1})/\pi(M_i).
\]
Indeed, the image of $M_i$ is contained in $\pi(M_i)$, so the map is
well defined, and surjectivity is immediate. Hence
$\pi(M_{i+1})/\pi(M_i)$ is a quotient, possibly zero, of
$M_{i+1}/M_i$. By Proposition~\ref{prop::filtration_Verma_supermodule},
\[
M_{i+1}/M_i\simeq M_0(\Lambda-\gamma)
\]
for some $\gamma\in\Gamma$. The assertion follows.
\end{proof}

Under the additional $\kk^{\CC}$-dominance every simple highest weight \(\gg_{\bar0}\)-composition factor is of the form \(L_0(\Lambda-\gamma)\) for some \(\gamma\in\Gamma\).

\begin{proposition}[{\cite[Theorem~2.5 and Corollary~2.7]{jakobsen1994full}}]
 Assume that every even composition factor in the $\even$-filtration of $L(\Lambda)$ has $\kk^{\CC}$-dominant integral highest weight. Then each simple factor is isomorphic to $L_{0}(\Lambda-\gamma)$ for some $\gamma \in \Gamma$.
\end{proposition}

For finite-dimensional or unitarizable $L(\Lambda)$, this is a direct sum.

\begin{corollary}\label{corr::decomp_even} If $L(\Lambda)$ is finite-dimensional or unitarizable, then it decomposes as a finite direct sum of simple $\even$-modules, each of highest weight $\Lambda - \gamma$ for some $\gamma \in \Gamma$.
\end{corollary} 

\begin{proof}
 Assume that $L(\Lambda)$ is finite-dimensional. Since $L(\Lambda)$ is a weight module for $\hh$, the action of $\hh$, and hence of the center $\mathfrak z(\even)\subseteq\hh$, is semisimple. Moreover, the derived algebra $[\even,\even]$ is semisimple. Therefore, by Weyl's complete reducibility theorem, the restriction of $L(\Lambda)$ to $[\even,\even]$ is completely reducible. Since the center acts semisimply and commutes with $[\even,\even]$, it follows that $L(\Lambda)$ is completely reducible as a $\even$-module. Hence
$
L(\Lambda)\vert_{\even}
$
is a finite direct sum of simple $\even$-modules. By the preceding proposition, each simple constituent is of the form $L_0(\Lambda-\gamma)$ for some $\gamma\in\Gamma$.

Assume that $L(\Lambda)$ is unitarizable, and let $\langle\cdot,\cdot\rangle$ be its positive-definite $\omega$-contravariant Hermitian form. By Corollary~\ref{cor::filtration_simple_supermodule}, $L(\Lambda)$ has finite length as a $\even$-module. Hence it suffices to show that every $\even$-submodule $W\subseteq L(\Lambda)$ is a direct summand. Since $W$ is a weight submodule and the weight spaces of $L(\Lambda)$ are finite-dimensional and mutually orthogonal,
\begin{equation}
L(\Lambda)=W\oplus W^\perp.
\end{equation}
Moreover, $W^\perp$ is a $\even$-submodule: for $X\in\even$, $v\in W^\perp$, and $w\in W$,
\begin{equation}
\langle Xv,w\rangle=\langle v,\omega(X)w\rangle=0,
\end{equation}
since $\omega(X)\in\even$ and $\omega(X)w\in W$. Thus $L(\Lambda)$ is semisimple as a $\even$-module. By the preceding proposition, every simple $\even$-constituent is isomorphic to $L_0(\Lambda-\gamma)$ for some $\gamma\in\Gamma$. This completes the proof.
\end{proof}

Our classification method depends on the action of the even quadratic Casimir $\Omega_{\even}$ on the $\even$-constituents of $L(\Lambda)$. We now express its eigenvalue in a form that will be useful below and substantially simplify the Dirac criterion (see Section~\ref{sec::Dirac_criterion}). The relevant notion is atypicality. Let $\Lambda\in\hh^{\ast}$ be a highest weight relative to $\Delta^+$, and let $\rho$ be the corresponding Weyl vector. The \emph{degree of atypicality} $\operatorname{atyp}(\Lambda)$ is the largest integer $r\in\ZZ_{\geq0}$ for which there exist pairwise orthogonal, linearly independent odd isotropic roots $\alpha_1,\ldots,\alpha_r\in\Delta_{\bar1}^{+}$ satisfying $(\Lambda+\rho,\alpha_i)=0$ for all $1\leq i\leq r$. The weight $\Lambda$ and the corresponding simple highest weight representation are called \emph{typical} if $\operatorname{atyp}(\Lambda)=0$ and \emph{atypical} otherwise. For $\gg=\mathfrak{sl}(n|n)$, linear independence in the
definition of atypicality is understood for the corresponding
lifted roots in $\mathfrak d^*$.

\begin{proposition}\label{prop::vorstufe}
Let $\Lambda$ be atypical of degree $k$, and let
$
X=\{\alpha_1,\ldots,\alpha_k\}\subseteq\Delta_{\bar1}^{+}
$
be a set of mutually orthogonal isotropic roots such that
$
(\Lambda+\rho,\alpha_i)=0$ for $i=1,\ldots,k.$
Then, for every $\gg_{\bar0}$-composition factor
$L_0(\nu)$ of $L(\Lambda)$, there exists
$
S\subseteq\Delta_{\bar1}^{+}\setminus X
$
such that $L_0(\nu)$ is a composition factor of
$
M_0\left(\Lambda-\Gamma_S\right).
$
\end{proposition}

\begin{proof}
By \cite[Theorem~6.1]{arXiv:1710.10528}, there exists a nonzero factor module
$M^{X}(\Lambda)$ of $M(\Lambda)$ with character
\[
\operatorname{ch} M^{X}(\Lambda)
=
e^{\Lambda}
\frac{
\prod_{\beta\in\Delta_{\bar1}^{+}\setminus X}(1+e^{-\beta})
}{
\prod_{\gamma\in\Delta_{\bar0}^{+}}(1-e^{-\gamma})
}.
\]
Since $M^{X}(\Lambda)$ is a nonzero factor module of the Verma module
$M(\Lambda)$, the image of the highest weight vector is nonzero.
Thus $M^{X}(\Lambda)$ is a highest weight module of highest weight
$\Lambda$, and hence
$
M^{X}(\Lambda)\twoheadrightarrow L(\Lambda).
$ Expanding the odd product gives
\[
\operatorname{ch} M^{X}(\Lambda)
=
\sum_{S\subseteq\Delta_{\bar1}^{+}\setminus X}
\operatorname{ch}
M_0\left(\Lambda-\Gamma_S\right).
\]
Since every $\gg_{\bar0}$-Verma module has finite length, Proposition~\ref{prop::filtration_Verma_supermodule} implies that $M(\Lambda)$ has finite length as a $\gg_{\bar0}$-module. Hence its quotient $M^X(\Lambda)$ also has finite $\gg_{\bar0}$-length. Therefore, in the Grothendieck group of
$\gg_{\bar0}$-modules,
\[
[M^{X}(\Lambda)]
=
\sum_{S\subseteq\Delta_{\bar1}^{+}\setminus X}
\left[
M_0\left(\Lambda-\Gamma_S\right)
\right].
\]
Since $L(\Lambda)$ is a quotient of $M^{X}(\Lambda)$, every
$\gg_{\bar0}$-composition factor $L_0(\nu)$ of $L(\Lambda)$
is a composition factor of $M^{X}(\Lambda)$. Hence
\[
\left[
M_0\left(\Lambda-\Gamma_S\right):L_0(\nu)
\right]>0
\]
for some
$S\subseteq\Delta_{\bar1}^{+}\setminus X$, which proves the claim.
\end{proof}

\begin{corollary}\label{cor::main_fd}
Let $\Lambda$ be atypical of degree $k$, and let
$
X=\{\alpha_1,\ldots,\alpha_k\}\subseteq\Delta_{\bar1}^{+}
$
be a set of mutually orthogonal isotropic roots such that
$
(\Lambda+\rho,\alpha_i)=0
$
for $i=1,\ldots,k$. Then, for every $\gg_{\bar0}$-composition factor $L_0(\nu)$ of $L(\Lambda)$, there exists
$
S\subseteq\Delta_{\bar1}^{+}\setminus X
$
such that
\[
(\nu+2\rho_{\bar{0}},\nu)
=
(\Lambda-\Gamma_S+2\rho_{\bar{0}},\Lambda-\Gamma_S).
\]
\end{corollary}

\begin{proof}
By the preceding proposition, $L_0(\nu)$ is a composition factor of $M_0(\Lambda-\Gamma_S)$ for some $S\subseteq\Delta_{\bar1}^{+}\setminus X$. Since $\Omega_{\even}$ is central, it acts on $M_0(\Lambda-\Gamma_S)$, and hence on every composition factor, by the scalar
$
(\Lambda-\Gamma_S+2\rho_{\bar0},\Lambda-\Gamma_S).
$
On the other hand, it acts on $L_0(\nu)$ by
$
(\nu+2\rho_{\bar0},\nu).
$
The two scalars are therefore equal.
\end{proof}

For simplicity, assume $pq=0$. Recall that in this case we fix the standard positive system, which we denote by $\Delta^{+}$ throughout the rest of this section. We extend Proposition~\ref{prop::vorstufe} and its corollary to highest weights of the form
\begin{equation}\label{eq::form_lambda_fd}
\Lambda=(\lambda^{1},\ldots,\lambda^{m}\vert
\mu^{1},\ldots,\mu^{k_{0}-1},\mu,\ldots,\mu),
\qquad
2\leq k_{0}\leq n,
\qquad
\mu^{k_{0}-1}\neq\mu.
\end{equation}
As above, the idea is to realize $L(\Lambda)$ as a quotient of a highest weight supermodule whose $\even$-composition factors can be controlled explicitly. For this purpose, we replace the Verma supermodule by a module obtained by parabolic induction from a suitable parabolic subalgebra. We first introduce some notation.

For $1\leq i\leq m$ and $1\leq k\leq n$, consider the ordering
\begin{equation}
(\epsilon_1,\ldots,\epsilon_{i-1},\delta_1,\ldots,\delta_{k-1},
\epsilon_i,\delta_k,\ldots,\delta_n,\epsilon_{i+1},\ldots,\epsilon_m),
\end{equation}
and denote the associated Borel subalgebra by $\bb(i,k)$, the associated positive system by $\Delta(i,k)^+$, and the associated simple system by $\Uppi(i,k)$. The consecutive block
$
(\epsilon_i,\delta_k,\ldots,\delta_n)
$
determines
\begin{equation}
J(i,k)
=
\{\epsilon_i-\delta_k,\,
\delta_k-\delta_{k+1},\ldots,
\delta_{n-1}-\delta_n\}.
\end{equation}
Set
$
\Delta_{J(i,k)}=\Delta\cap\ZZ J(i,k)
$
and
$
\Delta_{J(i,k)}^+=\Delta(i,k)^+\cap\ZZ J(i,k).
$
We define
\begin{equation}
\mathfrak l_{J(i,k)}
=
\mathfrak h
\oplus
\bigoplus_{\alpha\in\Delta_{J(i,k)}}\gg^\alpha,
\qquad
\mathfrak u_{J(i,k)}^+
=
\bigoplus_{\alpha\in
\Delta(i,k)^+\setminus\Delta_{J(i,k)}^+}
\gg^\alpha,
\end{equation}
and
\begin{equation}
\pp_{J(i,k)}
=
\mathfrak l_{J(i,k)}
\oplus
\mathfrak u_{J(i,k)}^+=
\bb(i,k)
+
\bigoplus_{\alpha\in\Delta_{J(i,k)}^+}
\gg^{-\alpha} .
\end{equation}
$\pp_{J(i,k)}$ is a parabolic subalgebra with Levi subalgebra $\ll_{J(i,k)}$ and nilpotent radical $\uu_{J(i,k)^{+}}$.

Consider $\Lambda$ of the form~\eqref{eq::form_lambda_fd} and assume
$
(\Lambda+\rho,\epsilon_i-\delta_k)=0
$
for some $1\leq i\leq m$ and $k_0\leq k\leq n$. Set
\begin{equation}\label{eq::definition_lambda_prime}
\Lambda'
\coloneqq
\Lambda-
\Gamma_{\Delta_{\bar1}^{+}
\cap(-\Delta(i,k)_{\bar1}^{+})}= \Lambda-\sum_{\alpha \in \Delta_{\bar1}^{+}
\cap(-\Delta(i,k)_{\bar1}^{+})}\alpha.
\end{equation}
\begin{lemma}\label{lemm::action_lambda_prime}
One has
$$
\Lambda'([\mathfrak l_{J(i,k)},\mathfrak l_{J(i,k)}]\cap\mathfrak h)=0.
$$
\end{lemma}

\begin{proof}
$
[\mathfrak l_{J(i,k)},\mathfrak l_{J(i,k)}]\cap\mathfrak h
$
is generated by the Cartan elements corresponding to
$
\delta_l-\delta_{l+1}$ for $k\leq l<n$,
and
$
\epsilon_i-\delta_k.
$ Moreover, 
$
\Delta_{\bar1}^{+}\cap\bigl(-\Delta(i,k)_{\bar1}^{+}\bigr)
=
\{\epsilon_a-\delta_b:a>i,\ 1\leq b\leq n\}
\cup
\{\epsilon_i-\delta_b:b<k\}
$
so that $\Lambda'$ has coordinates
$$
(\lambda')^i=\lambda^i-k+1, \qquad 
\mu'^l=\mu^l+m-i=\mu+m-i, \qquad k\leq l\leq n
$$
Since $\mu'^l=\mu'^{l+1}$ for $k\leq l<n$, one has
$
\Lambda'(h_{\delta_l-\delta_{l+1}})=0.
$
On the other hand,
\[
\Lambda'(h_{\epsilon_i-\delta_k})
=
(\lambda')^i+\mu'^k
=
\lambda^i+\mu+m-i-k+1
=
(\Lambda+\rho,\epsilon_i-\delta_k)
=
0.
\]
Thus $\Lambda'$ vanishes on a spanning set of
$
[\mathfrak l_{J(i,k)},\mathfrak l_{J(i,k)}]\cap\mathfrak h,
$
and hence
$
\Lambda'([\mathfrak l_{J(i,k)},\mathfrak l_{J(i,k)}]\cap\mathfrak h)=0.
$
\end{proof}

For $\Lambda'$, define the one-dimensional $\pp_{J(i,k)}$-module $\CC_{\Lambda'}$ by
\[
H\cdot1=\Lambda'(H),\qquad
\mathfrak u_{J(i,k)}^{+}\cdot1=0,
\]
and let $[\mathfrak l_{J(i,k)},\mathfrak l_{J(i,k)}]$ act trivially. This is well defined by Lemma~\ref{lemm::action_lambda_prime}. We define the parabolically induced module
\begin{equation}
M_{\pp_{J(i,k)}}(\Lambda')
\coloneqq
\UE(\gg)\otimes_{\UE(\pp_{J(i,k)})}\CC_{\Lambda'}.
\end{equation}

\begin{lemma}\label{lemm::parabolically_induced_module}
$L(\Lambda)$ is a $\gg$-subquotient of
$M_{\pp_{J(i,k)}}(\Lambda')$.
\end{lemma}

\begin{proof}
Let
$
v_{\Lambda'}
$
be the highest weight vector of $M_{\pp_{J(i,k)}}(\Lambda')$. Let
$
\mathfrak v=\bigoplus_{\alpha\in\Delta_{\bar1}^{+}}\gg^\alpha
$
and
$
\mathfrak c=\mathfrak v\cap\pp_{J(i,k)}.
$
Since $\mathfrak v$ is purely odd and abelian, the PBW theorem gives as $\bb_{\bar0}$-modules
$$
\UE(\mathfrak v)v_{\Lambda'}
\cong
\bigwedge(\mathfrak v/\mathfrak c)\otimes\CC_{\Lambda'}.
$$
The weights of 
$\mathfrak v/\mathfrak c$ are precisely the roots in $\Delta_{\bar{1}}^{+}\cap (-\Delta(i,k)^{+}_{\bar{1}})$.

Choose nonzero root vectors $X_\alpha\in\gg^\alpha$ for $\alpha\in \Delta_{\bar{1}}^{+}\cap (-\Delta(i,k)^{+}_{\bar{1}})$
and set
$$
w=\prod_{\alpha\in \Delta_{\bar{1}}^{+}\cap (-\Delta(i,k)^{+}_{\bar{1}})}X_\alpha v_{\Lambda'},
$$
where the product is taken in the fixed order on $\Delta_{\bar{1}}^{+}\cap (-\Delta(i,k)^{+}_{\bar{1}})$. Under the above
identification, $w$ spans the top exterior power of
$\mathfrak v/\mathfrak c$. Hence $w\neq0$ and $w$ has weight $\Lambda'+\Gamma_{\Delta_{\bar{1}}^{+}\cap (-\Delta(i,k)^{+}_{\bar{1}})}=\Lambda$.

The statement follows if we show that $\nn_{\bar{0}}^{+}w=0$ and $\nn_{\bar{1}}^{+}w=0$. Since \(\mathfrak v/\mathfrak c\) is a finite-dimensional \(\mathfrak b_{\bar0}\)-module, its top exterior power is one-dimensional. Every \(X\in\nn_{\bar0}^{+}\) acts nilpotently on \(\mathfrak v/\mathfrak c\), and hence acts on its top exterior power by the scalar \(0\). Therefore \(\nn_{\bar0}^{+}w=0\).
It is also annihilated by
$\nn_{\bar1}^{+}=\mathfrak v$: the action of $\mathfrak v/\mathfrak c$
is by exterior multiplication and therefore vanishes on the top
exterior power, while $\mathfrak c$ annihilates $v_{\Lambda'}$ by the definition
of the inducing $\pp_{J(i,k)}$-module. Thus
$
\nn^{+}w=0.
$
\end{proof}

\begin{proposition}\label{prop::atypical_composition_factors_degree_k}
Assume $p=0$ or $q=0$ and
$\Lambda=(\lambda^{1},\ldots,\lambda^{m}\vert
\mu^{1},\ldots,\mu^{k_{0}-1},\mu,\ldots,\mu)$ with
$\mu^{k_{0}-1}\neq\mu$. Suppose that
$(\Lambda+\rho,\epsilon_{i}-\delta_{k})=0$ for some
$k_{0}\leq k\leq n$. Then, for every $\even$-composition factor
$L_{0}(\nu)$ of $L(\Lambda)$, there exists
$S\subseteq\Delta_{\bar1}^{+}$ such that $L_{0}(\nu)$ is a
composition factor of $M_{0}(\Lambda-\Gamma_{S})$ and
$\epsilon_{i}-\delta_{l}\notin S$ for all $k\leq l\leq n$.
\end{proposition}

\begin{proof}
Suppose $(\Lambda+\rho,\epsilon_i-\delta_k)=0$. As above, associate to it
\[
\Lambda'
\coloneqq
\Lambda-
\Gamma_{\Delta_{\bar1}^{+}\cap(-\Delta(i,k)_{\bar1}^{+})}
\]
and construct $M_{\pp_{J(i,k)}}(\Lambda')$. By
Lemma~\ref{lemm::parabolically_induced_module}, $L(\Lambda)$ is a
$\gg$-subquotient, and hence a $\even$-subquotient, of
$M_{\pp_{J(i,k)}}(\Lambda')$. It therefore suffices to determine the
possible $\even$-composition factors of
$M_{\pp_{J(i,k)}}(\Lambda')$.

Consider the filtration of $M_{\pp_{J(i,k)}}(\Lambda')$ by odd degree.
By the PBW theorem, the associated graded $\even$-module is
\[
\operatorname{gr}M_{\pp_{J(i,k)}}(\Lambda')
\cong
\UE(\even)\otimes_{\UE((\pp_{J(i,k)})_{\bar0})}
\left(
\bigwedge\left(
\gg_{\bar1}/(\pp_{J(i,k)})_{\bar1}
\right)
\otimes\CC_{\Lambda'}
\right).
\]
The ordering defining $\bb(i,k)$ preserves the relative order of the
$\epsilon$'s and of the $\delta$'s, and hence
$
\bb_{\bar0}\subseteq(\pp_{J(i,k)})_{\bar0}.
$
Therefore each homogeneous component of the exterior algebra admits a finite $\bb_{\bar0}$-stable weight filtration. By exactness of induction from $\bb_{\bar0}$, the corresponding induced $\even$-module has a finite filtration by $\even$-Verma modules. Since the homogeneous component of $\operatorname{gr}M_{\pp_{J(i,k)}}(\Lambda')$ is a quotient of this induced module, each of its simple $\even$-composition factors is a composition factor of one of these $\even$-Verma modules.

The weights of
$
\gg_{\bar1}/(\pp_{J(i,k)})_{\bar1}
$
are precisely
\[
\bigl(\Delta_{\bar1}^{+}\cap(-\Delta(i,k)_{\bar1}^{+})\bigr)
\sqcup
-
\left(
\bigl(\Delta_{\bar1}^{+}\cap\Delta(i,k)_{\bar1}^{+}\bigr)
\setminus
\{\epsilon_i-\delta_l:k\leq l\leq n\}
\right).
\]
Consequently, every $\even$-Verma module occurring in the resulting
filtration has highest weight of the form
$
\Lambda-\Gamma_S
$
for some
$
S\subseteq\Delta_{\bar1}^{+}
$
such that
$
\epsilon_i-\delta_l\notin S$ for all $k\leq l\leq n$. 
Hence every simple $\even$-composition factor of
$M_{\pp_{J(i,k)}}(\Lambda')$ is a composition factor of
$M_0(\Lambda-\Gamma_S)$ for some such $S$.

Since $L(\Lambda)$ is a $\even$-subquotient of
$M_{\pp_{J(i,k)}}(\Lambda')$, every $\even$-composition factor
$L_0(\nu)$ of $L(\Lambda)$ is therefore a composition factor of
$M_0(\Lambda-\Gamma_S)$ for some
$
S\subseteq\Delta_{\bar1}^{+}
$
with
$
\epsilon_i-\delta_l\notin S
$
for all $k\leq l\leq n$.
\end{proof}

\vspace{0.1cm}
\begin{remark}~\label{rmk::extend_atypicality}
The same argument applies simultaneously to several atypical roots. More precisely, suppose that
$
(\Lambda+\rho,\epsilon_{i_a}-\delta_{j_a})=0$ for $a=1,\ldots,r$,
and set
\[
X=\bigcup_{a=1}^{r}
\{\epsilon_{i_a}-\delta_l:j_a\leq l\leq n\}.
\]
Then, for every $\even$-composition factor $L_0(\nu)$ of $L(\Lambda)$, there exists a subset
$
S\subseteq\Delta_{\bar1}^{+}\setminus X
$
such that $L_0(\nu)$ is a composition factor of
$
M_0(\Lambda-\Gamma_{S}).
$
\end{remark}

\begin{corollary}\label{cor::simplification_action_casimir_fd}
 Let $\Lambda$ be of the form~\eqref{eq::form_lambda_fd} and assume $(\Lambda+\rho, \epsilon_{i}-\delta_{k})=0$ for some $k_{0}\leq k \leq n$. Then for every $\even$-composition factor $L_{0}(\nu)$ of $L(\Lambda)$, there exists $S \subset \Delta_{\bar{1}}^{+}\setminus \{\epsilon_{i}-\delta_{k}, \ldots, \epsilon_{i}-\delta_{n}\}$ such that 
 $$
(\nu+2\rho_{\bar{0}},\nu)=(\Lambda-\Gamma_{S}+2\rho_{\bar{0}}, \Lambda-\Gamma_{S}).
 $$
\end{corollary}

Similarly, one can treat the case $pq\neq 0$. We consider $\Lambda \in \hh^{\ast}$ of the form 
\begin{equation}\label{eq::general_form_lambda_ifd} \Lambda = (\lambda, \ldots, \lambda, \lambda^{i_{0}+1}, \ldots, \lambda^{p}, \lambda^{p+1}, \ldots, \lambda^{m-j_{0}-1}, \lambda', \ldots, \lambda'\vert \mu^{1}, \ldots, \mu^{n}), \end{equation}
with $\lambda^{i_{0}+1}\neq \lambda$ and $\lambda^{m-j_{0}-1}\neq \lambda'$. We omit the proof, since the argument is identical.

\begin{proposition}\label{prop::action_Casimir_ifd} Assume $pq\neq 0$ and $\Lambda$ is of the form~\eqref{eq::general_form_lambda_ifd}. Let $L_{0}(\nu)$ be a $\even$-composition factor of $L(\Lambda)$. 
\begin{enumerate}
 \item[(a)] If $(\Lambda+\rho, \epsilon_{i}-\delta_{k})=0$ for some $1 \leq i \leq i_{0}$ and $1\leq k \leq n$, then there exists some $S \subset \Delta_{\bar{1}}^{+}$ such that $L_{0}(\nu)$ is a composition factor of $M_{0}(\Lambda-\Gamma_{S})$ and $\epsilon_{1}-\delta_{k},\ldots, \epsilon_{i}-\delta_{k}\notin S$.
 \item[(b)] If $(\Lambda+\rho, -\epsilon_{m-j}+\delta_{k})=0$ for some $0 \leq j \leq j_{0}$ and $1\leq k \leq n$, then there exists some $S \subset \Delta_{\bar{1}}^{+}$ such that $L_{0}(\nu)$ is a composition factor of $M_{0}(\Lambda-\Gamma_{S})$ and $-\epsilon_{m-j}+\delta_{k},\ldots, -\epsilon_{m}+\delta_{k}\notin S$.
\end{enumerate}
 In particular, in either case, 
 \[
 (\nu+2\rho_{\bar{0}},\nu)=(\Lambda-\Gamma_{S}+2\rho_{\bar{0}}, \Lambda-\Gamma_{S}).
 \]
\end{proposition}

\begin{remark} The argument for Proposition~\ref{prop::action_Casimir_ifd} applies simultaneously to several atypical roots.
\end{remark}
\subsection{General theory of unitarizable supermodules}
\noindent We briefly recall the definition of a unitarizable supermodule and record its immediate consequences.

A \emph{super 
pre-Hilbert space} is a $\ZZ_{2}$-graded complex pre-Hilbert space $(\mathcal{H}=\mathcal{H}_{\bar{0}}
\oplus \mathcal{H}_{\bar{1}}, \langle \cdot, 
\cdot \rangle_{\mathcal{H}})$ such that $\mathcal{H}_{\bar{0}}$ and $\mathcal{H}_{\bar{1}}$ are mutually 
orthogonal subspaces of $\mathcal{H}$ with respect to the inner product $\langle \cdot, 
\cdot \rangle_{\mathcal{H}}$.\footnote{Pre-super Hilbert spaces have an equivalent description in terms of even super positive definite super Hermitian forms. For a detailed comparison, we refer to \cite{jakobsen1994full}.} The inner product is conjugate-linear in the first and linear in the second argument. With this in place, we define \emph{unitarizable supermodules} relative to a real form of $\gg$. We realize the real forms as $\mathfrak{g}^{\omega} \coloneqq \{X \in \mathfrak{g} : \omega(X) = -i^{p(X)}X\}$, where $\omega$ is some conjugate-linear anti-involution.

\begin{definition}[{\cite[Definition 2.3]{jakobsen1994full}}]
Let $\HH$ be a $\gg$-supermodule, and let $\omega$ be a conjugate-linear anti-involution on $\gg$. The supermodule $\HH$ is called an $\omega$\emph{-unitarizable} $\gg$\emph{-supermodule} if $\HH$ is a super pre-Hilbert space such that for all $v,w \in \HH$ and all $X \in \gg$, the Hermitian product $\langle \cdot,\cdot \rangle_{\HH}$ is $\omega$-contravariant:
$$
\langle Xv,w \rangle_{\HH} = \langle v,\omega(X)w \rangle_{\HH}.
$$
\end{definition}

If we work with $\UE(\gg)$-supermodules instead, we extend $\omega$ to $\UE(\gg)$ in the natural way, using the same notation. In this context, a $\gg$-supermodule $\HH$ is unitarizable if and only if it is a Hermitian representation over $(\UE(\gg), \omega)$, meaning that $\langle Xv, w \rangle_{\HH} = \langle v, \omega(X)w \rangle_{\HH}$ holds for all $v, w \in \HH$ and $X \in \UE(\gg)$. When $\omega$ is to be implied from context, we just say ``unitarizable''.

We fix a real form $\mathfrak{g}^{\omega}$ and let $\mathcal{H}$ be an $\omega$-unitarizable $\mathfrak{g}$-supermodule. The first consequence of unitarity is complete reducibility in the following sense.

\begin{lemma}
The orthogonal complement $U^{\perp}$ of every $\mathfrak{g}$-invariant subspace $U$ of $\mathcal{H}$ is again $\mathfrak{g}$-invariant. Moreover, if $\HH$ is simple, it decomposes as an algebraic direct sum $$
\HH=U\oplus U^{\perp}.
$$
\end{lemma}

\begin{proof}
Let $U\subseteq\mathcal{H}$ be $\mathfrak{g}$-invariant, $v\in U^{\perp}$, and $X\in\mathfrak{g}$. For every $u\in U$, unitarity gives
\[
\langle Xv,u\rangle=\langle v,\omega(X)u\rangle=0,
\]
since $\omega(X)u\in U$. Hence $Xv\in U^{\perp}$. 

Since $\HH$ is simple, it is $\hh$-semisimple with finite-dimensional weight spaces. Moreover, since $U$ is $\mathfrak h$-stable, it decomposes into weight spaces,
\[
U=\bigoplus_\mu U^\mu,
\qquad
U^\mu=U\cap\HH^\mu.
\]
Distinct weight spaces of $\HH$ are orthogonal. Since each $\HH^\mu$ is finite-dimensional, one has
\[
\HH^\mu
=
U^\mu\oplus (U^\mu)^\perp.
\]
Taking the algebraic direct sum over all weights yields
$
\HH
=
U\oplus U^\perp.
$
\end{proof}

Next consider the underlying $\mathfrak{g}_{\bar{0}}$-module $\mathcal{H}_{\mathrm{ev}}$, obtained by restricting $\mathcal{H}$ to $\mathfrak{g}_{\bar{0}}$ and forgetting the $\mathbb{Z}_{2}$-grading.

\begin{lemma}
The module $\mathcal{H}_{\mathrm{ev}}$ is unitarizable with respect to the real form $\mathfrak{g}_{\bar{0}}^{\omega}$ and is completely reducible as a $\mathfrak{g}_{\bar{0}}$-module in the sense above.
\end{lemma}

\begin{proof}
The restriction of the Hermitian form on $\mathcal{H}$ to $\mathcal{H}_{\mathrm{ev}}$ is positive definite and satisfies the same invariance relation for $\mathfrak{g}_{\bar{0}}$. The assertion therefore follows from the preceding lemma.
\end{proof}

Together with complete reducibility, the following theorem reduces the study of unitarizable supermodules to highest and lowest weight supermodules.

\begin{theorem}[\cite{furutsu1991classification,neeb2011lie}]
\label{thm::HW_property_M}
The special linear Lie superalgebra $\gg$ admits non-trivial $\omega$-unitarizable supermodules if and only if the conjugate-linear anti-involution $\omega$ is associated with one of the real forms
\[
\mathfrak{su}(p,q\vert n,0)
\quad \text{or} \quad
\mathfrak{su}(p,q\vert 0,n),
\qquad p+q=m.
\]
If $pq\neq0$, the non-trivial simple unitarizable supermodules for $\su(p,q\vert 0,n)$ are of highest weight type, whereas those for $\su(p,q\vert n,0)$ are of lowest weight type.
\end{theorem}

\begin{remark}\label{rmk::symmetry_HW_LW}
The second part of the theorem reflects the following duality (\emph{cf.}~\cite{Schmidt_unitary_branching}). There is a one-to-one correspondence between unitarizable highest weight supermodules of highest weight $\Lambda$ with respect to a conjugate-linear anti-involution $\omega$ and unitarizable lowest weight supermodules of lowest weight $-\Lambda$ with respect to the sign-twisted anti-involution $\omega^{\vee}=(-1)^{p(\cdot)}\omega$. More precisely, let $L(\Lambda)$ be a highest weight supermodule with finite-dimensional weight spaces. Its contragredient dual is
\[
L(\Lambda)^{\vee}
\coloneqq
\bigoplus_{\mu\in\mathcal{P}_{L(\Lambda)}}
\bigl(L(\Lambda)^{\mu}\bigr)^{*},
\]
where $\mathcal{P}_{L(\Lambda)} \coloneqq \{\mu \in \hh^{\ast} : L(\Lambda)^{\mu}\neq \{0\}\}$. For homogeneous $x\in\gg$ and $f\in L(\Lambda)^{\vee}$, the contragredient action is
\[
(xf)(v)
=
-(-1)^{p(x)p(f)}f(xv).
\]
Then $L(\Lambda)^{\vee}$ is a lowest weight supermodule of lowest weight $-\Lambda$, and $L(\Lambda)$ is $\omega$-unitarizable if and only if $L(\Lambda)^{\vee}$ is $\omega^{\vee}$-unitarizable.
\end{remark}

Since every highest or lowest weight $\gg$-supermodule has finite $\even$-length, we obtain the following corollary.

\begin{corollary}\label{cor::finite_and_infinite_dimensionality}
\begin{enumerate}
\item[(a)] Every simple $\omega$-unitarizable $\gg$-supermodule is a finite direct sum of simple $\mathfrak{g}_{\bar0}$-modules.
\item[(b)] A non-trivial simple unitarizable supermodule is finite-dimensional if and only if $pq=0$, and infinite-dimensional if and only if $pq\neq0$.
\end{enumerate}
\end{corollary}

\begin{proof}
It remains to prove~(b). By~(a), every simple $\omega$-unitarizable $\gg$-supermodule $\HH$ is a finite direct sum of simple unitarizable $\even$-modules. Suppose first that $pq=0$. Then the corresponding real form of $\even$ is compact. Hence every simple unitarizable $\even$-module is finite-dimensional, and therefore so is $\HH$. This proves one implication of~(b).

Suppose now that $pq\neq0$ and that $\HH$ is finite-dimensional. Its restriction to the noncompact simple factor $\su(p,q)$ is a finite-dimensional unitarizable representation and hence trivial. Therefore its complexification $\sl(m)$ acts trivially on $\HH$. Let $\pi:\gg\to\End(\HH)$ denote the representation. Since
$
[\sl(m),\odd]=\odd,
$
we obtain
\[
\pi(\odd)
=
\pi([\sl(m),\odd])
=
[\pi(\sl(m)),\pi(\odd)]
=
0.
\]
Since
$
[\odd,\odd]=\even,
$
it follows that $\pi(\even)=0$. Thus $\gg$ acts trivially on $\HH$. Hence a non-trivial simple unitarizable $\gg$-supermodule cannot be finite-dimensional when $pq\neq0$. This proves the converse implication and completes the proof of~(b).
\end{proof}

For the remainder, we consider only $\omega$-unitarizable highest weight $\mathfrak{g}$-supermodules, since the lowest weight case is analogous. If $pq\neq0$, we fix the real form $\su(p,q\vert 0,n)$ and the associated conjugate-linear anti-involution $\omega_{(-,+)}$, and write $\su(p,q\vert n)$ for $\su(p,q\vert 0,n)$. If $p=0$ or $q=0$, we set $\su(m\vert n)\coloneqq\su(m,0\vert n,0)$ and fix $\omega_{+}$. The case of $\omega_{-}$ is analogous, so we state only the corresponding results. Henceforth, we omit the prefix $\omega_{+}/\omega_{(-,+)}$ unless needed and simply speak of unitarizable highest weight $\mathfrak{g}$-supermodules.

We are now in a position to state the main problem of this work.

\begin{problem}
Describe the full set of unitarizable simple $\gg$-supermodules. 
Equivalently, determine all highest weights 
$\Lambda \in \hh^{\ast}$ for which there exists a unitarizable 
simple $\gg$-supermodule of highest weight $\Lambda$.
\end{problem}

A first step towards a complete description consists in restricting the possible form of the occurring highest weights, which follows directly from the definition of unitarity. We refer to these as the \emph{unitarity conditions}. These conditions are the starting point of our classification.

\begin{proposition}\label{prop::unitarity_conditions}
Let $(\HH,\langle\cdot,\cdot\rangle)$ be a unitarizable highest weight $\gg$-supermodule with highest weight $\Lambda=(\lambda^{1},\ldots,\lambda^{m}\vert\mu^{1},\ldots,\mu^{n})\in\hh^{\ast}$. Then the following \emph{unitarity conditions} hold:
\begin{enumerate}
\item[(a)] If $pq=0$, then $\Lambda$ is the highest weight of a finite-dimensional simple $\even$-module and 
\begin{enumerate}
 \item[(i)] 
$
\lambda^{1}\ge\cdots\ge\lambda^{m}\ge-\mu^{n}\ge\cdots\ge-\mu^{1};
$
\item[(ii)] Either
$
(\Lambda+\rho,\epsilon_m-\delta_n)>0,
$
or there exists a unique $1\leq a\leq n$ such that
$
(\Lambda+\rho,\epsilon_m-\delta_a)=0,
$
in which case
$
\mu^a=\mu^{a+1}=\cdots=\mu^n.
$
\end{enumerate}
\item[(b)] If $pq\neq0$, then $\Lambda$ is the highest weight of a unitarizable highest weight $\even$-module and
\begin{enumerate}
\item[(i)]
$
\lambda^{p+1}\ge\cdots\ge\lambda^{m}\ge-\mu^{n}\ge\cdots\ge-\mu^{1}\ge\lambda^{1}\ge\cdots\ge\lambda^{p};
$
\item[(ii)] Either
$
(\Lambda+\rho,\epsilon_1-\delta_1)<0,
$
or there exists a unique $1\leq i\leq p$ such that
$
(\Lambda+\rho,\epsilon_i-\delta_1)=0,
$
in which case
$
\lambda^1=\cdots=\lambda^i.
$
\item[(iii)] Either
$
(\Lambda+\rho,-\epsilon_m+\delta_n)<0,
$
or there exists a unique $p+1\leq i\leq m$ such that
$
(\Lambda+\rho,-\epsilon_i+\delta_n)=0,
$
in which case
$
\lambda^i=\cdots=\lambda^m.
$
\end{enumerate}
\end{enumerate}
\end{proposition}

\begin{proof}
The proofs of~(a) and~(b) are analogous, so we prove only~(a). Let $v_{\Lambda}$ be the highest weight vector of $\HH$. Then $v_{\Lambda}$ also generates a unitarizable highest weight $\even$-module of highest weight $\Lambda$. For an odd root $\alpha$, let $\gg^{\alpha}$ denote the corresponding root space. If $e_{\alpha}\in\gg^{\alpha}$ is a root vector of weight $\alpha$, then $\omega(e_{\alpha})=e_{-\alpha}\in\gg^{-\alpha}$ since $\omega=\omega_{+}$.

Fix an odd positive root $\alpha \coloneqq \epsilon_{i} - \delta_{a} \in \Delta_{\bar{1}}^{+}$. Since $\langle\cdot,\cdot\rangle$ is positive definite, one has
\[
0 \leq \langle e_{-\alpha}v_{\Lambda}, e_{-\alpha}v_{\Lambda} \rangle
= \langle v_{\Lambda}, \omega(e_{-\alpha}) e_{-\alpha} v_{\Lambda} \rangle
= \Lambda([\omega(e_{-\alpha}), e_{-\alpha}]) \langle v_{\Lambda}, v_{\Lambda} \rangle
= (\lambda^{i} + \mu^{a}) \langle v_{\Lambda}, v_{\Lambda} \rangle.
\]
Since $\langle v_{\Lambda}, v_{\Lambda} \rangle > 0$, it follows that $\lambda^{i} \geq -\mu^{a}$ for all $1 \leq i \leq m$ and $1 \leq a \leq n$. By Corollary~\ref{corr::decomp_even}, $\Lambda$ is dominant integral, \emph{i.e.}, $
\lambda^{i} - \lambda^{i+1} \in \ZZ_{\geq 0}$ and $\mu^{a} - \mu^{a+1} \in \ZZ_{\geq 0}.
$
Thus,
\[
\lambda^{1} \geq \cdots \geq \lambda^{m} \geq -\mu^{n} \geq \cdots \geq -\mu^{1}.
\]

Assume now that
$(\Lambda+\rho,\epsilon_m-\delta_a)=0$ for some $1\leq a\leq n$. Set
\[
\beta_b\coloneqq \epsilon_m-\delta_b,
\qquad
v_r\coloneqq e_{-\beta_1}\cdots e_{-\beta_r}v_\Lambda,
\qquad
v_0\coloneqq v_\Lambda.
\]
For $b<a$, since $\mu^b\geq\mu^a$, we have
\[
(\Lambda+\rho,\beta_b)
=
(\Lambda+\rho,\beta_a)
+(\Lambda,\delta_a-\delta_b)+(a-b)>0.
\]
For $1\leq r<a$, one has $(\rho,\beta_r)=1-r$. Thus,
\[
\langle v_r,v_r\rangle
=
(\Lambda-\beta_1-\cdots-\beta_{r-1},\beta_r)
\langle v_{r-1},v_{r-1}\rangle
=
(\Lambda+\rho,\beta_r)
\langle v_{r-1},v_{r-1}\rangle.
\]
Hence
\[
\langle v_{a-1},v_{a-1}\rangle
=
\prod_{b=1}^{a-1}
(\Lambda+\rho,\beta_b)
\langle v_\Lambda,v_\Lambda\rangle
>0.
\]
Positive-definiteness now gives
\[
\begin{aligned}
0
&\leq
\langle e_{-\beta_n}v_{a-1},
 e_{-\beta_n}v_{a-1}\rangle \\
&=
(\Lambda-\beta_1-\cdots-\beta_{a-1},\beta_n)
\langle v_{a-1},v_{a-1}\rangle \\
&=
\bigl(
(\Lambda+\rho,\beta_a)
+
(\Lambda,\delta_a-\delta_n)
\bigr)
\langle v_{a-1},v_{a-1}\rangle \\
&=
(\Lambda,\delta_a-\delta_n)
\langle v_{a-1},v_{a-1}\rangle.
\end{aligned}
\]
Since
$
(\Lambda,\delta_a-\delta_n)
=
-\mu^a+\mu^n
\leq 0
$
and $\langle v_{a-1},v_{a-1}\rangle>0$, we conclude that
$
(\Lambda,\delta_a-\delta_n)=0.
$

Finally, assume that $(\Lambda+\rho,\beta_a)\neq0$ for all $a=1,\ldots,n$, where $\beta_a=\epsilon_m-\delta_a$. By the norm identity established above, $\langle v_a,v_a\rangle=(\Lambda+\rho,\beta_a)\langle v_{a-1},v_{a-1}\rangle$ for $a=1,\ldots,n$. For $a=1$, positive-definiteness gives $0\leq\langle v_1,v_1\rangle=(\Lambda+\rho,\beta_1)\langle v_\Lambda,v_\Lambda\rangle$. Since $(\Lambda+\rho,\beta_1)\neq0$ and $\langle v_\Lambda,v_\Lambda\rangle>0$, we obtain $(\Lambda+\rho,\beta_1)>0$, and hence $\langle v_1,v_1\rangle>0$. Suppose inductively that $\langle v_{a-1},v_{a-1}\rangle>0$. Then $0\leq\langle v_a,v_a\rangle=(\Lambda+\rho,\beta_a)\langle v_{a-1},v_{a-1}\rangle$, and the assumption $(\Lambda+\rho,\beta_a)\neq0$ implies $(\Lambda+\rho,\beta_a)>0$. Thus $\langle v_a,v_a\rangle>0$. By induction, $(\Lambda+\rho,\epsilon_m-\delta_a)>0$ for all $a=1,\ldots,n$. This proves~(ii) and completes the proof.
\end{proof}

\subsubsection{Unitarizable highest weight \texorpdfstring{$\even$}{}-modules} \label{subsubsec::unitarizable_even_HWM}
Any unitarizable highest weight $\gg$-supermodule of highest weight $\Lambda$ decomposes as a direct sum of unitarizable highest weight $\even$-modules. Hence $\Lambda$ is itself the highest weight of a unitarizable highest weight $\even$-module. The structure of these supermodules is as follows. 

One has the decomposition
$\even=\L\oplus\R\oplus\u(1)^{\CC}$,
with $\L=\su(p,q)^{\CC}$ and $\R=\su(n)^{\CC}$.
 We denote the associated root systems of $(\L, \hh\vert_{\L})$ and $(\R,\hh\vert_{\R})$ by $\Delta_{\L}$ and $\Delta_{\R}$, respectively. In particular, $\Delta_{\bar{0}} = \Delta_{\L}\sqcup \Delta_{\R}$. The positive systems are $\Delta_{\L}^{+} \coloneqq \Delta^{+}_{\bar{0}} \cap \Delta_{\L}$ and $\Delta_{\R}^{+} \coloneqq \Delta^{+}_{\bar{0}} \cap \Delta_{\R}$. The highest weight of a unitarizable simple $\gg$-supermodule is of the form $\nu=(\nu^{\L}\vert\nu^{\R})$, where, in standard coordinates, $\nu^{\L}=(\lambda^{1},\ldots,\lambda^{m})$ and $\nu^{\R}=(\mu^{1},\ldots,\mu^{n})$. Here $\nu^{\L}$ is the highest weight of a unitarizable $\L$-module, while $\nu^{\R}$ is the highest weight of a finite-dimensional simple $\R$-module. In particular, one obtains the following description.

\begin{lemma} \label{lemm::structure_even_module}
 Any unitarizable simple highest weight $\even$-module $L_{0}(\nu)$ is given by the outer tensor product of unitarizable simple highest weight $\su(p,q)^{\CC}$-, $\su(n)^{\CC}$-, and $\u(1)^{\CC}$-modules, respectively, \emph{i.e.},
 $$
 L_{0}(\nu) \cong L_{0}(\nu^{\L};\L) \boxtimes L_{0}(\nu^{\R};\R) \boxtimes \CC_{\nu}.
 $$
\end{lemma}

All such modules will be listed in Section~\ref{subsec::parametrization_ifd}, providing a parametrization of the highest weights of unitarizable highest weight $\even$-modules.

A characteristic property of unitarizable highest weight supermodules is that they belong to the class of \emph{Harish-Chandra modules}. We begin by recalling the definition.

Set $\kk \coloneqq \mathfrak{s}(\mathfrak{u}(p)\oplus \mathfrak{u}(q))\oplus \su(n)\oplus \u(1)$. Then $\kk$ is a maximal compact subalgebra of $\su(p,q\vert n)_{\bar{0}}$, with $\kk = \su(p,q\vert n)_{\bar{0}}$ if $p=0$ or $q=0$. One has the \emph{equal rank condition}
\begin{equation} \label{eq::equal_rank_condition}
 \hh \subset \kk^{\CC} \subset \even \subset \gg,
\end{equation}
where $\kk^{\CC}$ denotes the complexification of $\kk$. The root system $\Delta_{c}$ associated with $(\hh,\kk^{\CC})$ is
\begin{equation}
\begin{aligned}
\Delta_{c}\coloneqq{}&
\{\pm(\epsilon_i-\epsilon_j):1\leq i<j\leq p\}
\cup
\{\pm(\epsilon_i-\epsilon_j):p+1\leq i<j\leq m\} \\
&\cup
\{\pm(\delta_a-\delta_b):1\leq a<b\leq n\},
\end{aligned}
\end{equation}
which forms a subset of $\Delta_{\bar{0}}$. Thus, a root $\alpha \in \Delta_{\bar{0}}$ is termed \emph{compact} if $\alpha \in \Delta_{c}$, or equivalently, if the corresponding root vector belongs to $\kk^{\CC}$; otherwise, it is called \emph{non-compact}. 
The Weyl group associated with $\Delta_{c}$ will be denoted by $W_{c}$, which is indeed a subgroup of $W$.
The set $\Delta_{n} \coloneqq \Delta_{\bar{0}} \setminus \Delta_{c}$ will be called the \emph{set of non-compact roots}, so that we have the decomposition:
\begin{equation}
\Delta_{\bar{0}} = \Delta_{c} \sqcup \Delta_{n}.
\end{equation}
Moreover, we define the \emph{set of positive compact} and \emph{positive non-compact roots} by $\Delta_{c,n}^{+} \coloneqq \Delta^{+} \cap \Delta_{c,n}$. The associated Weyl vectors are given by $\rho_{c,n} \coloneqq \frac{1}{2} \sum_{\alpha \in \Delta_{c,n}^{+}} \alpha$, and with respect to $\Delta_{c}^{+}$, an element $\lambda \in \hh^{\ast}$ is called $\Delta_{c}^{+}$\emph{-dominant integral} if it satisfies the following conditions:
\begin{equation} \label{eq::dominant_integral}
 \begin{cases}
 (\lambda + \rho_{c}, \alpha) \in \ZZ_{> 0} \ &\text{for all} \ \alpha \in \{\epsilon_{i} - \epsilon_{j} : 1 \leq i < j \leq p \ \text{or} \ p+1 \leq i < j \leq m\}, \\
 (\lambda + \rho_{c}, \alpha) \in \ZZ_{< 0} \ &\text{for all} \ \alpha \in \{\delta_{a} - \delta_{b} : 1 \leq a < b \leq n\}.
 \end{cases}
\end{equation}
The dominant integral weights are in one-to-one correspondence with simple $\kk^{\CC}$-modules.

Every $\even$-module is a $\kk^{\CC}$-module. A $(\even,\kk^{\CC})$-module is called a \emph{Harish-Chandra module} if it is finitely generated as a $\UE(\even)$-module and locally finite as a $\kk^{\CC}$-module. By~\cite[Lemma~IX.3.10]{NeebHW}, any unitarizable highest weight $\even$-module is a Harish-Chandra module.

Since every unitarizable highest weight $\gg$-supermodule is a Harish-Chandra module, it is natural to choose positive systems compatible with the $\kk^{\CC}$-module structure of the odd part. This decomposition also explains the choice of positive systems above. If $pq\neq 0$, the odd part decomposes into four simple $\kk^{\CC}$-constituents, and the standard and nonstandard positive systems select different halves of this decomposition. If $p=0$ or $q=0$, by contrast, the odd part has only two simple $\kk^{\CC}$-constituents.

The study of unitarizability, and in particular the classification of unitarizable supermodules, is closely tied to the algebraic Dirac operator. Introduced in \cite{huang2005dirac}, this operator provides the main tool for the unitarity analysis developed in \cite{SchmidtDirac}. The next section recalls the corresponding Dirac formalism and sharpens the characterization of unitarizability through the Dirac inequality.

\section{The Dirac criterion} \label{sec::Dirac_criterion}
We briefly introduce the cubic Dirac operator, following \cite{huang2005dirac,Schmidt_perturbations}, and relate it to unitarizable highest weight supermodules as in \cite{SchmidtDirac}. This yields Dirac inequalities, which give necessary and sufficient conditions for unitarity.

Assume first that $\gg\neq\sl(n\vert n)$, and equip $\gg$ with the normalized invariant form $B(X,Y)\coloneqq\frac12(X,Y)$. Its restriction to $\odd$ is symplectic. Fix a polarization $\odd=\nn_{\bar1}^{+}\oplus\nn_{\bar1}^{-}$ and dual bases $\{\partial_i\}_{i=1}^{mn}$ and $\{x_i\}_{i=1}^{mn}$ such that
\begin{equation}
B(\partial_i,x_j)=-B(x_j,\partial_i)=\frac12\delta_{ij}.
\end{equation}

The associated Weyl algebra $\mathscr W(\gg_{\bar1})$ is the algebra of differential operators with polynomial coefficients in $x_1,\ldots,x_{mn}$, where $\partial_i$ acts as $\frac{\partial}{\partial x_i}$. It is generated by $x_i$ and $\partial_i$ subject to
$[x_i,x_j]_{\mathscr W}=[\partial_i,\partial_j]_{\mathscr W}=0$
and
$[\partial_i,x_j]_{\mathscr W}=\delta_{ij}$
for $1\leq i,j\leq mn$. 

There is a natural injective Lie algebra homomorphism $\alpha : \even \to \Weyl$ given by 
\begin{multline} \label{eq::alpha_explicit}
\alpha(X)=\sum_{i,j=1}^{mn}\bigl(B(X,[\partial_i,\partial_j])x_ix_j+B(X,[x_i,x_j])\partial_i\partial_j\bigr)\\
-\sum_{i,j=1}^{mn}2B(X,[x_i,\partial_j])x_j\partial_i-\sum_{l=1}^{mn}B(X,[\partial_l,x_l]).
\end{multline}
We denote the associated diagonal embedding by $\Updelta\colon\even\to\UE(\gg)\otimes\mathscr W(\odd)$, that is,
$\Updelta(x)=x\otimes1+1\otimes\alpha(x)$. Let $\Omega_{\gg}$ and $\Omega_{\even}$ denote the quadratic Casimir elements of $\gg$ and $\even$, respectively, and let $\Omega_{\even,\Delta}$ denote the image of $\Omega_{\even}$ under $\Updelta$. We denote the subalgebra of $\even$-invariants in $\UE(\gg)\otimes\mathscr W(\odd)$ by $\WW(\gg,\even)$.

Kostant's Dirac operator for $(\gg,\even)$ is
\begin{equation}\label{eq::Dirac_operator}
\Dirac_{\gg,\even}\coloneqq
2\sum_{i=1}^{mn}(\partial_i\otimes x_i-x_i\otimes\partial_i)
\in\UE(\gg)\otimes\mathscr W(\odd).
\end{equation}
In analogy with the case of reductive Lie algebras, it has the following properties. 

\begin{theorem}[\cite{huang2005dirac}]\label{thm::properties_D}
 The relative Dirac element $\Dirac_{\gg,\even}$ has the following properties.
 \begin{enumerate}
 \item[(a)] It is independent of the chosen polarization.
 \item[(b)] It is invariant under the diagonal action of $\even$, and hence
 $\Dirac_{\gg,\even}\in\mathcal W(\gg,\even)$.
 \item[(c)] It satisfies the Parthasarathy square formula
 \begin{equation*}
 \Dirac_{\gg,\even}^{2}
 =
 -\Omega_{\gg}\otimes1
 +\Omega_{\even,\Delta}
 -\frac18\operatorname{tr}_{\odd}
 \bigl(\Omega_{\even}\bigr).
 \end{equation*}
 \end{enumerate}
\end{theorem}

For a $\gg$-supermodule $M$, $\Dirac_{\gg,\even}$ acts as an operator
\begin{equation}
\Dirac_{\gg,\even}\in\End_{\CC}\bigl(M\otimes M(\odd)\bigr),
\end{equation}
where $M(\odd)=\CC[x_{1},\ldots,x_{mn}]$ is the oscillator module for $\Weyl$. Here $x_i$ acts by multiplication and $\partial_i$ by differentiation. In particular, $M(\odd)$ is a simple $\Weyl$-module. Via the embedding $\alpha: \even \to \Weyl$, we view $M(\odd)$ also as a $\even$-module. The vacuum vector $1$ is annihilated by all $\partial_i$ and satisfies
\begin{equation}
\alpha(H)1=-\rho_{\bar1}(H)1,\qquad H\in\hh.
\end{equation}
Especially, $M(\odd)$ is a weight module whose set of weights is contained in
$-\rho_{\bar{1}}-\ZZ_{+}\Delta_{\bar{1}}^{+}$.

The simple $\Weyl$-module $M(\odd)$ carries a unique positive definite Hermitian form $\langle\cdot,\cdot\rangle_{M(\odd)}$, known as the \emph{Bargmann--Fock Hermitian form} or \emph{Fischer--Fock Hermitian form} \cite{BargmannHermitianForm, FischerHermitianForm, FockHermitianForm}, such that $\partial_{k}$ and $x_{k}$ are adjoint to each other and the following orthogonality relations hold:
\begin{equation}
\begin{split}
\langle \prod_{k=1}^{mn}x_{k}^{p_{k}},\prod_{k=1}^{mn}x_{k}^{q_{k}}\rangle_{M(\odd)}=\begin{cases}\prod_{k=1}^{mn}p_{k}! \qquad &\text{if} \ p_{k}=q_{k} \ \text{for all} \ k, \\
0 \qquad &\text{otherwise}.\end{cases}
\end{split}
\end{equation}

Let $M$ be a highest weight supermodule. As a $\even$-module, $M\otimes M(\odd)$ is a weight module on which $\nn_{\bar{0}}^{+}$ acts locally finitely. Its simple subquotients are therefore highest weight $\even$-modules, which we denote by $L_0(\nu)$. The square formula in Theorem~\ref{thm::properties_D} then determines the action of $\Dirac_{\gg,\even}^{2}$ on each such composition factor.

\begin{theorem}\label{thm::dirac-square-action}
Let $M$ be a simple highest weight $\gg$-representation of highest weight $\Lambda$, and let $L_{0}(\nu)$ be a $\even$-composition factor of $M\otimes M(\odd)$ of highest weight $\nu$. Then $\Dirac_{\gg,\even}^{2}$ acts on the composition factor $L_{0}(\nu)$ by the scalar
\begin{equation*}
2((\nu+\rho_{\bar{1}}+2\rho,\nu+\rho_{\bar{1}})-(\Lambda+2\rho,\Lambda)).
\end{equation*}
\end{theorem}

We next use this action to derive an inequality for unitarizable supermodules. The argument depends on the choice of positive system, since $M(\odd)$ is not unitarizable as a $\even$-module in general. A necessary and sufficient condition is that $\Delta^{+}$ be $\omega$-adapted \cite{Schmidt_unitary_branching}. We call $\Delta^{+}$ \emph{$\omega$-adapted} if there exist dual bases
$\{\partial_i\}_{i=1}^{mn}$ of $\nn_{\bar1}^{+}$ and
$\{x_i\}_{i=1}^{mn}$ of $\nn_{\bar1}^{-}$, such that
\begin{equation}
(\partial_i,x_j)=\delta_{ij},
\qquad
\omega(x_i)=\varepsilon\partial_i
\end{equation}
for all $i,j$, with $\varepsilon\in\{\pm1\}$ independent of $i$. Equivalently, the Hermitian form
$
h_\omega(x,y)\coloneqq(\omega(x),y)
$
on $\nn_{\bar1}^{-}$ is definite; indeed,
$
h_\omega(x_i,x_j)=\varepsilon\delta_{ij}.
$
As shown in \cite{Schmidt_unitary_branching}, $\Delta_{\st}^{+}$ is $\omega_{\pm}$-adapted, whereas $\Delta_{\nst}^{+}$ is $\omega_{(+,-)}$- and $\omega_{(-,+)}$-adapted, and not conversely. If $(M,\langle\cdot,\cdot\rangle_M)$ is a unitarizable highest weight $\gg$-supermodule with respect to $\omega$, and $\Delta^{+}$ is $\omega$-adapted, then Theorem~\ref{thm::dirac-square-action} yields a Dirac inequality. 

To this end, we equip $M\otimes M(\odd)$ with the tensor-product Hermitian form
\begin{equation}
\langle v\otimes P,w\otimes Q\rangle_{M\otimes M(\odd)}
=
\langle v,w\rangle_M\langle P,Q\rangle_{M(\odd)}.
\end{equation}
The induced adjoint on $U(\gg)\otimes\mathscr W(\odd)$ is
\begin{equation}
(u\otimes a)^{\ast}
=
\omega(u)\otimes a^{\dagger},
\end{equation}
where $x_i^{\dagger}=\partial_i$ and $\partial_i^{\dagger}=x_i$. 

\begin{theorem}[{\cite{Schmidt_unitary_branching}}]\label{thm::properties_Dirac_operator}
Let $(M,\langle\cdot,\cdot\rangle_M)$ be a unitarizable highest weight representation, and suppose that $\Delta^{+}$ is $\omega$-adapted with sign $\varepsilon\in\{\pm1\}$. Then:
\begin{enumerate}
\item[(a)] $M\otimes M(\odd)$ is a unitarizable $\even$-representation with respect to $\langle\cdot,\cdot\rangle_{M\otimes M(\odd)}$ and is therefore completely reducible.
\item[(b)] One has
\begin{equation*}
\Dirac_{\gg,\even}^{\ast}
=
-\varepsilon\Dirac_{\gg,\even}.
\end{equation*}
Thus $\Dirac_{\gg,\even}$ is selfadjoint for $\varepsilon=-1$ and skewadjoint for $\varepsilon=+1$. Consequently, $\Dirac_{\gg,\even}^{2}$ is positive semidefinite for $\varepsilon=-1$ and negative semidefinite for $\varepsilon=+1$.
\item[(c)] Suppose that
$
M\vert_{\even}
=
\bigoplus_{\nu}L_{0}(\nu)^{m_{\Lambda}(\nu)}.
$
Then, for every $\nu$ with $m_{\Lambda}(\nu)\neq0$, the Dirac inequality holds
\begin{equation*}
-\varepsilon\bigl(-(\Lambda+2\rho,\Lambda)+(\nu+2\rho,\nu)\bigr)\geq0.
\end{equation*}
Moreover, the inequality is strict if $\nu\neq\Lambda$.
\end{enumerate}
\end{theorem}

The Dirac inequality yields a complete characterization of unitarity in terms of the highest weights of the simple composition factors in the $\even$-filtration of $M$ described in Section~\ref{subsubsec::even_filtration}. 

\begin{theorem}[\cite{SchmidtDirac}] \label{thm::unitarity_via_Dirac_inequality} Let $M$ be a simple highest weight $\gg$-supermodule with highest weight $\Lambda$. 
Then $M$ is unitarizable if and only if the following two conditions are satisfied:
 \begin{enumerate}
 \item[(a)] $\Lambda$ is the highest weight of a unitarizable highest weight $\even$-module.
 \item[(b)] If $L_{0}(\nu)$ is a simple $\even$-composition factor in the $\even$-filtration of $M$ and $\nu \neq \Lambda$, then
\begin{equation*}
 \begin{aligned}
 (\nu + 2\rho, \nu) \begin{cases}> (\Lambda+2\rho,\Lambda) \qquad & pq\neq 0, \\
 <(\Lambda+2\rho, \Lambda) \qquad &pq=0. \end{cases}
 \end{aligned}
\end{equation*}
 \end{enumerate}
\end{theorem}

\begin{proof}
A proof is given in~\cite{SchmidtDirac}. For completeness, we give a different proof. We treat the case $pq\neq0$. The case $pq=0$ is analogous.

If $\langle\cdot,\cdot\rangle$ denotes the Shapovalov form on $M$ and $v$ is a $\even$-highest weight vector of highest weight $\nu$, then a direct calculation gives
\begin{equation}\label{eq:review-norm}
2\bigl(-(\Lambda+2\rho,\Lambda)+(\nu+2\rho,\nu)\bigr)\langle v,v\rangle_{M}=\langle \Dirac_{\gg,\even}^{2}(v\otimes 1),v\otimes 1\rangle_{M\otimes M(\odd)}=-4\sum_{i=1}^{mn}\langle x_{i}\partial_{i}v,v\rangle_{M}.
\end{equation}
Of particular interest is the operator $T\coloneqq-4\sum_{i=1}^{mn}x_i\partial_i$. Recall the non-compact even positive roots $\Delta_n^{+}$ from Section~\ref{subsubsec::unitarizable_even_HWM}, and set $\pp^{\pm}\coloneqq\sum_{\alpha\in\Delta_n^{+}}\gg^{\pm\alpha}$. Then,
\begin{equation}\label{eq:review-commutator}
[T,\kk^{\CC}]=0,\qquad [T,x]=-4nx\quad\text{for all }x\in\pp^+.
\end{equation}
In the compact case $\pp^+=0$, and $T$ commutes with the entire even subalgebra.

Suppose first that $M$ is unitarizable. Then $M$ is a unitarizable highest weight supermodule, hence simple, and decomposes into a finite direct sum of $\even$-modules $L_{0}(\nu)$. In particular, its highest weight vector $v_{\Lambda}$ generates a unitarizable highest weight $\even$-module of highest weight $\Lambda$, yielding~\textup{(a)}. For a nonzero highest weight vector $v$ in a non-top even constituent, some $\partial_i v$ is nonzero. Otherwise, $v$ would be a non-top highest weight vector of the simple supermodule $M$. Equation~\eqref{eq:review-norm} now gives~\textup{(b)}.

Conversely, assume~\textup{(a)} and~\textup{(b)}. Since $M$ is simple, the Shapovalov form is non-degenerate. We normalize it so that $\langle v_{\Lambda},v_{\Lambda}\rangle_{M}=1$. Give negative odd root vectors degree $1$, negative non-compact even root vectors degree $2$, and $\kk$ degree $0$, and put $v_\Lambda$ in degree $0$. These assignments are compatible with the Lie brackets, since~\cite{SchmidtDirac}
\begin{equation*}
\begin{gathered}
[\kk,\kk]\subseteq\kk,\qquad
[\kk,\nn_{\bar1}^{-}]\subseteq\nn_{\bar1}^{-},\qquad
[\kk,\mathfrak p^{-}]\subseteq\mathfrak p^{-},\\
[\nn_{\bar1}^{-},\nn_{\bar1}^{-}]=\mathfrak p^{-},\qquad
[\mathfrak p^{-},\nn_{\bar1}^{-}]=0,\qquad
[\mathfrak p^{-},\mathfrak p^{-}]=0.
\end{gathered}
\end{equation*}
Hence they induce, by PBW, a grading
\[
M=\bigoplus_{r\geq0}M[r],
\qquad
M[0]=\UE(\kk^{\CC})v_\Lambda.
\]
Since $\omega$ reverses degree, the contravariant form satisfies $\langle M[r],M[s]\rangle=0$ whenever $r\neq s$. Moreover,
\[
\partial_iM[r]\subseteq M[r-1],\qquad \pp^+M[r]\subseteq M[r-2].
\]
Condition~\textup{(a)} implies that $\langle\cdot,\cdot\rangle_{M}$ is positive semidefinite on $M[0]=U(\kk^{\CC})v_\Lambda$. Since $M[0]$ is orthogonal to all other degrees and $\langle\cdot,\cdot\rangle_{M}$ is non-degenerate, it is in fact positive definite on $M[0]$. Moreover, every $M[r]$ is finite-dimensional.

We prove by induction on $r$ that $\langle\cdot,\cdot\rangle_{M}$ is positive definite on $M[r]$. The case $r=0$ is settled. Assume the assertion for all degrees $<r$. First,
\begin{equation}\label{eq:joint-kernel}
\bigcap_i\ker(\partial_i\vert_{M[r]})=0.
\end{equation}
Indeed, the common kernel is $\kk$-stable and, since $[\nn_{\bar1}^+,\nn_{\bar1}^+]=\pp^+$, is annihilated by $\pp^+$. A nonzero $\kk$-highest weight vector in it would therefore be annihilated by all positive root spaces of $\mathfrak g$, giving a non-top highest weight vector in the simple module $M$. Define
\[
K(u,v)\coloneqq4\sum_{i=1}^{mn}\langle\partial_i u,\partial_i v\rangle_{M}=\langle Tu,v\rangle_{M},\qquad u,v\in M[r].
\]
By the induction hypothesis and~\eqref{eq:joint-kernel}, $K$ is positive definite. Since $T^*=T$, the operator $T\vert_{M[r]}$ is self-adjoint with respect to $K$ and hence diagonalizable with real nonzero eigenvalues. Suppose that $t<0$ is an eigenvalue, with eigenspace $V_t$. For $x\in\pp^+$ and $v\in V_t$, \eqref{eq:review-commutator} gives
\[
T(xv)=(t-4n)xv.
\]
Since $xv\in M[r-2]$ and $\langle\cdot,\cdot\rangle_{M}$ is positive there by induction,
\[
0\leq\langle Txv,xv\rangle_{M}=(t-4n)\langle xv,xv\rangle_{M}.
\]
As $t-4n<0$, this forces $xv=0$. Hence $\pp^+V_t=0$. Since $V_t$ is a finite-dimensional $\kk$-stable weight space, it contains a nonzero $\kk$-highest weight vector $v\in V_t$ of weight $\nu$. Then $v$ is an actual $\mathfrak g_{\bar0}$-highest weight vector. Since $r>0$, $\nu\neq\Lambda$, and the simple quotient $L_0(\nu)$ of $U(\mathfrak g_{\bar0})v$ is an even composition factor of $M$. Therefore, by~\textup{(b)},
\[
-(\Lambda+2\rho,\Lambda)+(\nu+2\rho,\nu)>0.
\]
On the other hand,
\[
K(v,v)=\langle Tv,v\rangle_{M}=t\langle v,v\rangle_{M},
\]
so $\langle v,v\rangle_{M}<0$. Applying~\eqref{eq:review-norm} gives
\[
2(-(\Lambda+2\rho,\Lambda)+(\nu+2\rho,\nu))=t<0,
\]
a contradiction. Hence every eigenvalue of $T\vert_{M[r]}$ is positive.

Finally, since $T$ is positive with respect to $K$,
\[
\langle u,u\rangle_{M}=K(T^{-1}u,u)>0\qquad\text{for all }u\in M[r]\setminus\{0\}.
\]
Thus $\langle\cdot,\cdot\rangle_{M}$ is positive definite on $M[r]$, completing the induction. Since different degrees are orthogonal, $\langle\cdot,\cdot\rangle_{M}$ is positive definite on all of $M$. Hence $M$ is unitarizable.
\end{proof}

Theorem~\ref{thm::unitarity_via_Dirac_inequality} is the basis of our classification of unitarity. The key point is that the problem reduces to the inequalities associated with the weights $\Lambda-\alpha$, where $\alpha\in\Delta_{\bar1}^{+}$. 

\begin{lemma}\label{lemm::Dirac_basic_const}
Let $\nu=\Lambda-\alpha$ with $\alpha\in\Delta_{\bar1}^{+}$. Then
\[
(\nu+2\rho,\nu)\ge(\Lambda+2\rho,\Lambda)\quad\Longleftrightarrow\quad(\Lambda+\rho,\alpha)\le0.
\]
\end{lemma}

\begin{proof}
Since every odd root of $\mathfrak{sl}(m\vert n)$ is isotropic, one has
$
(\alpha,\alpha)=0.
$
Hence
\[
\begin{aligned}
(\nu+2\rho,\nu)-(\Lambda+2\rho,\Lambda)
&=
(\Lambda-\alpha+2\rho,\Lambda-\alpha)
-(\Lambda+2\rho,\Lambda)\\
&=
-2(\Lambda+\rho,\alpha)+(\alpha,\alpha)\\
&=
-2(\Lambda+\rho,\alpha).
\end{aligned}
\]
Therefore
$
(\nu+2\rho,\nu)\geq(\Lambda+2\rho,\Lambda)
$
if and only if
$
(\Lambda+\rho,\alpha)\leq0.
$
\end{proof}

Motivated by Corollary~\ref{cor::simplification_action_casimir_fd} and Proposition~\ref{prop::action_Casimir_ifd}, we can give the following useful sufficient condition for unitarity. 

\begin{proposition}\label{prop::continuous_parameter}
Let $M$ be a simple highest weight $\gg$-supermodule of highest weight $\Lambda\in\hh^{*}$, and assume that $\Lambda$ is the highest weight of a unitarizable highest weight $\even$-module. Suppose that for every highest weight $\nu\neq\Lambda$ of a $\even$-composition factor of $M$, there exists
$
\gamma=\gamma_{1}+\cdots+\gamma_{k},
$
where the $\gamma_i$ are pairwise distinct positive odd roots, such that $L_{0}(\nu)$ is a composition factor of $M_{0}(\Lambda-\gamma)$ and one of the following holds:
\begin{enumerate}
\item[(a)] $pq=0$, $M$ is finite-dimensional, and
\[
(\Lambda+\rho,\gamma_i)>0
\qquad
\text{for all }i=1,\ldots,k;
\]
\item[(b)] $pq\neq0$, $M$ is infinite-dimensional, and
\[
(\Lambda+\rho,\gamma_i)<0
\qquad
\text{for all }i=1,\ldots,k.
\]
\end{enumerate}
Then $M$ is unitarizable. In particular, $M$ is unitarizable if $(\Lambda+\rho,\alpha)>0$ for all $\alpha\in\Delta_{\bar1}^{+}$ in the finite-dimensional case, and if $(\Lambda+\rho,\alpha)<0$ for all $\alpha\in\Delta_{\bar1}^{+}$ in the infinite-dimensional case.
\end{proposition}

\begin{proof}
We verify the Dirac inequality for every highest weight
$\nu\neq\Lambda$ of a $\even$-composition factor of $M$.
The proposition will then follow from
Theorem~\ref{thm::unitarity_via_Dirac_inequality}.

By assumption, there exists
$
\gamma=\gamma_1+\cdots+\gamma_k,
$
where the $\gamma_i\in\Delta_{\bar1}^{+}$ are pairwise distinct and satisfy
\[
(\Lambda+\rho,\gamma_i)
\begin{cases}
>0,&\text{if $M$ is finite-dimensional},\\
<0,&\text{if $M$ is infinite-dimensional},
\end{cases}
\]
such that
\begin{equation}\label{eq::same_even_casimir}
(\nu+2\rho_{\bar0},\nu)
=
(\Lambda-\gamma+2\rho_{\bar0},\Lambda-\gamma).
\end{equation}

Set
$
\eta\coloneqq \Lambda-\gamma-\nu.
$
Since $L_0(\nu)$ is a composition factor of
$M_0(\Lambda-\gamma)$, we have
$\eta\in Q_{\bar0}^{+}$, where $Q_{\bar0}^{+}$ is the
nonnegative integer span of the positive even roots. Hence
\[
(\rho_{\bar1},\nu)
=
(\rho_{\bar1},\Lambda-\gamma)-(\rho_{\bar1},\eta).
\]
Using $\rho=\rho_{\bar0}-\rho_{\bar1}$ together with
\eqref{eq::same_even_casimir}, we obtain
\[
\begin{aligned}
(\nu+2\rho,\nu)
&=(\nu+2\rho_{\bar0},\nu)-2(\rho_{\bar1},\nu)\\
&=(\Lambda-\gamma+2\rho_{\bar0},\Lambda-\gamma)
 -2(\rho_{\bar1},\Lambda-\gamma)+2(\rho_{\bar1},\eta)\\
&=(\Lambda-\gamma+2\rho,\Lambda-\gamma)
 +2(\rho_{\bar1},\eta).
\end{aligned}
\]
Consequently,
\begin{equation}\label{eq::to_estimate}
\begin{aligned}
(\Lambda+2\rho,\Lambda)-(\nu+2\rho,\nu)
&=2(\Lambda+\rho,\gamma)-(\gamma,\gamma)
 -2(\rho_{\bar1},\eta).
\end{aligned}
\end{equation}

If $pq=0$, we use the standard positive system, for which
$\rho_{\bar1}$ is orthogonal to every even root. Thus
$(\rho_{\bar1},\eta)=0$. If $pq\neq0$, we use the nonstandard positive system, for which
\[
(\rho_{\bar1},\alpha)=
\begin{cases}
0,&\alpha\in\Delta_c^+,\\
n,&\alpha\in\Delta_n^+.
\end{cases}
\]
Therefore $(\rho_{\bar1},\eta)\ge0$, and
\[
(\nu+2\rho,\nu)-(\Lambda+2\rho,\Lambda)
\ge -2(\Lambda+\rho,\gamma)+(\gamma,\gamma).
\]
This is the lower bound needed in the infinite-dimensional case.
We estimate this expression one row at a time in the
finite-dimensional case, and one column at a time, separately
for each root type, in the infinite-dimensional case.
Within each group, dominance and integrality make the pairings
$(\Lambda+\rho,\gamma_i)$ differ by at least one. This gives
the required strict inequality for each group. The terms
between different groups only strengthen the inequality,
so adding the estimates proves the claim.

Suppose first that $M$ is finite-dimensional.
Here $\Delta^{+}$ is the standard positive system. For two
positive odd roots,
\[
(\epsilon_r-\delta_a,\epsilon_s-\delta_b)
=
\delta_{rs}-\delta_{ab}.
\]
Thus distinct roots in the same row have pairing \(1\), whereas roots in distinct rows have pairing \(0\) or \(-1\). Partitioning the selected roots into rows, it
therefore suffices to prove
\[
2(\Lambda+\rho,\sigma)-(\sigma,\sigma)>0
\]
for the sum $\sigma$ of the selected roots in each nonempty
row.

Fix such a row, containing $k'$ roots
$\gamma'_1,\ldots,\gamma'_{k'}$, ordered so that
\[
(\Lambda+\rho,\gamma'_1)
<
\cdots
<
(\Lambda+\rho,\gamma'_{k'}).
\]
Dominance and integrality for the $\mathfrak{sl}(n)$-factor
imply that consecutive values differ by at least one.
Indeed, for $a<b$,
\[
(\Lambda+\rho,\epsilon_r-\delta_a)
-
(\Lambda+\rho,\epsilon_r-\delta_b)
=
\mu^a-\mu^b+b-a
\geq b-a.
\]
Setting
$
\xi\coloneqq(\Lambda+\rho,\gamma'_1)>0,
$
we obtain
\[
(\Lambda+\rho,\gamma'_j)\geq\xi+j-1,
\qquad j=1,\ldots,k'.
\]
Since the roots are isotropic and the bilinear pairing of two distinct roots in the same row is $1$,
\[
\left(\sum_{j=1}^{k'}\gamma'_j,
\sum_{j=1}^{k'}\gamma'_j\right)
=
k'(k'-1).
\]
Consequently,
\[
\begin{aligned}
2\left(\Lambda+\rho,\sum_{j=1}^{k'}\gamma'_j\right)
-
\left(\sum_{j=1}^{k'}\gamma'_j,
\sum_{j=1}^{k'}\gamma'_j\right)
&\geq
2\sum_{j=1}^{k'}(\xi+j-1)-k'(k'-1)\\
&=2k'\xi>0.
\end{aligned}
\]
Summing over the nonempty rows and using
\eqref{eq::to_estimate}, we conclude that
\[
(\nu+2\rho,\nu)<(\Lambda+2\rho,\Lambda).
\]

Suppose now that $M$ is infinite-dimensional. Set
\[
A\coloneqq
\{\epsilon_r-\delta_a:1\leq r\leq p,\ 1\leq a\leq n\},
\qquad
B\coloneqq
\{-\epsilon_s+\delta_a:p<s\leq m,\ 1\leq a\leq n\},
\]
and write
\[
\gamma
=
\sum_{i=1}^{k_1}\alpha_i
+
\sum_{j=1}^{k_2}\beta_j,
\qquad
\alpha_i\in A,\quad \beta_j\in B.
\]
Since $(\alpha_i,\beta_j)\in\{0,1\}$,
\[
\begin{aligned}
(\nu+2\rho,\nu)-(\Lambda+2\rho,\Lambda)
&=-2(\Lambda+\rho,\gamma)+(\gamma,\gamma)
 +2(\rho_{\bar1},\eta)\\
&\geq -2(\Lambda+\rho,\gamma)+(\gamma,\gamma)\\
&\geq
-2\left(\Lambda+\rho,\sum_{i=1}^{k_1}\alpha_i\right)
+\left(\sum_{i=1}^{k_1}\alpha_i,
 \sum_{i=1}^{k_1}\alpha_i\right)\\
&\quad
-2\left(\Lambda+\rho,\sum_{j=1}^{k_2}\beta_j\right)
+\left(\sum_{j=1}^{k_2}\beta_j,
 \sum_{j=1}^{k_2}\beta_j\right).
\end{aligned}
\]

Within each of $A$ and $B$, the bilinear pairing of two distinct roots in the same column is $-1$, whereas roots in different columns have non-negative pairing. It therefore suffices to
prove
\[
-2(\Lambda+\rho,\sigma)+(\sigma,\sigma)>0
\]
for the sum $\sigma$ of the selected roots in each nonempty
column of either type.

Denote the roots in such a column by
$\gamma'_1,\ldots,\gamma'_{k'}$, ordered so that
\[
(\Lambda+\rho,\gamma'_1)
>
\cdots
>
(\Lambda+\rho,\gamma'_{k'}).
\]
Since $\Lambda$ is the highest weight of a unitarizable
highest weight $\even$-module, it is dominant integral for
the compact $\mathfrak{sl}(p)$- and $\mathfrak{sl}(q)$-factors.
Consequently, the displayed pairings differ by at least one.
Setting
$\xi\coloneqq(\Lambda+\rho,\gamma'_1)<0,
$
we have
\[
(\Lambda+\rho,\gamma'_j)\leq\xi-j+1,
\qquad j=1,\ldots,k'.
\]
Moreover,
\[
\left(\sum_{j=1}^{k'}\gamma'_j,
\sum_{j=1}^{k'}\gamma'_j\right)
=
-k'(k'-1).
\]
Hence
\[
\begin{aligned}
-2\left(\Lambda+\rho,\sum_{j=1}^{k'}\gamma'_j\right)
+
\left(\sum_{j=1}^{k'}\gamma'_j,
\sum_{j=1}^{k'}\gamma'_j\right)
&\geq
-2\sum_{j=1}^{k'}(\xi-j+1)-k'(k'-1)\\
&=-2k'\xi>0.
\end{aligned}
\]
Since $\nu\neq\Lambda$, at least one column is nonempty.
Summing these estimates gives
\[
(\nu+2\rho,\nu)>(\Lambda+2\rho,\Lambda).
\]
The Dirac criterion therefore yields unitarizability in
both cases.
\end{proof}

For $\gg=\sl(n\vert n)$, $n\geq2$, the supertrace form is
nondegenerate on $\odd$, so $\Dirac_{\gg,\even}$ is still
defined. Its square is computed in $\mathfrak{gl}(n\vert n)$,
using a compatible highest weight extension of $M$.
All weight pairings are then taken in $\dd^*$. Any change of the extended highest weight adds $c\chi$ to
every extended weight, where
\begin{equation}
\chi=\sum_{i=1}^{n}\epsilon_i-\sum_{a=1}^{n}\delta_a,
\qquad
\chi\vert_{\hh}=0,
\qquad
(\chi,\alpha)=0\quad(\alpha\in\Delta).
\end{equation}
Since $\widetilde\nu-\widetilde\Lambda$ lies in the root
lattice, this changes
$
(\widetilde\nu+2\rho,\widetilde\nu)
-
(\widetilde\Lambda+2\rho,\widetilde\Lambda)
$
by $2c(\widetilde\nu-\widetilde\Lambda,\chi)=0$.
Thus the square formula and the inequalities used in the
proofs are independent of the chosen extension.
Consequently, Theorem~\ref{thm::unitarity_via_Dirac_inequality}
and Proposition~\ref{prop::continuous_parameter} remain valid
for $\sl(n\vert n)$, with their quadratic expressions
interpreted using these compatible $\dd^*$-weights.

\section{Finite-dimensional classification}\label{sec::fd}
\subsection{Statement of the classification}
We classify the finite-dimensional non-trivial unitarizable simple $\mathfrak g$-supermodules. By Corollary~\ref{cor::finite_and_infinite_dimensionality}, these occur precisely when $pq=0$. The relevant real forms are $\su(m,0\vert n,0)\cong\su(0,m\vert 0,n)$, corresponding to $\omega_{+}$, and $\su(m,0\vert 0,n)\cong\su(0,m\vert n,0)$, corresponding to $\omega_{-}$, by Theorem~\ref{thm::HW_property_M}, where
\begin{equation}
\omega_{\pm}
\left(
\begin{array}{@{}c|c@{}}
A & B \\
\hline
C & D
\end{array}
\right)
=
\left(
\begin{array}{@{}c|c@{}}
A^{\dagger} & \pm C^{\dagger} \\
\hline
\pm B^{\dagger} & D^{\dagger}
\end{array}
\right).
\end{equation}
Moreover, as explained in Section~\ref{subsubsec::unitarizable_even_HWM}, these are defined with respect to the standard positive system $\Delta_{\st}^{+}$, which we denote simply by $\Delta^{+}$ throughout this section.

Before stating the classification theorem, we recall the full set of finite-dimensional simple $\gg$-supermodules. A weight $\lambda$ is \emph{dominant integral} if and only if
\begin{equation}
2\frac{(\lambda+\rho_{\bar{0}},\alpha)}{(\alpha,\alpha)}\in\ZZ_{>0}
\quad\text{for all }\alpha\in\Delta_{\bar{0}}^{+},
\end{equation}
equivalently, if there exists a finite-dimensional simple $\even$-module of highest weight $\lambda$ with respect to $\bb_{\bar{0}}$. We denote the set of dominant integral weights by $P^{++}$ and refer to its elements as $\Delta^{+}$-dominant integral weights. For $\lambda\in P^{++}$, let $L(\lambda)$ denote the simple supermodule of highest weight $\lambda$ with respect to $\bb$, with even highest weight vector. Then the finite-dimensional simple $\gg$-supermodules are parameterized by
\begin{equation}
\{\,L(\lambda),\Pi L(\lambda):\lambda\in P^{++}\,\}.
\end{equation}
The set $P^{++}$ also parametrizes the finite-dimensional simple
$\even$-modules. Those that are unitarizable with respect to
$\omega_{\pm}|_{\even}$ are precisely those whose highest weights
are symmetric, that is,
\begin{equation}
\Lambda(\omega_{\pm}(H))=\overline{\Lambda(H)}
\qquad\text{for every }H\in\hh.
\end{equation}
We denote this set of weights by
$
P^{++}_{\omega_{\pm}}.
$

\begin{theorem}\label{thm::finite-dimensional-classification}
\begin{enumerate}
\item[(a)]
The finite-dimensional simple $\omega_{+}$-unitarizable
$\gg$-supermodules are precisely $L(\Lambda)$ and $\Pi L(\Lambda)$,
where $\Lambda\in P^{++}_{\omega_{+}}$ satisfies either
\begin{enumerate}
\item[(1)] $(\Lambda+\rho,\epsilon_m-\delta_n)>0$, or
\item[(2)] $(\Lambda+\rho,\epsilon_m-\delta_k)=0$ and
$(\Lambda,\delta_k-\delta_n)=0$ for some $k\in\{1,\ldots,n\}$.
\end{enumerate}

\item[(b)]
The finite-dimensional simple $\omega_{-}$-unitarizable
$\gg$-supermodules are precisely $L(\Lambda)$ and $\Pi L(\Lambda)$,
where $\Lambda\in P^{++}_{\omega_{-}}$ satisfies either
\begin{enumerate}
\item[(1)] $(\Lambda+\rho,\epsilon_1-\delta_1)<0$, or
\item[(2)] $(\Lambda+\rho,\epsilon_i-\delta_1)=0$ and
$(\Lambda,\epsilon_1-\epsilon_i)=0$ for some $i\in\{1,\ldots,m\}$.
\end{enumerate}
\end{enumerate}
\end{theorem}

\begin{remark}
As explained in Remark~\ref{rmk::symmetry_HW_LW}, $L(\Lambda)$ is an $\omega_{\pm}$-unitarizable highest weight $\gg$-supermodule of highest weight $\Lambda$ if and only if its contragredient dual $L(\Lambda)^{\vee}$ is an $\omega_{\mp}$-unitarizable lowest weight $\gg$-supermodule of lowest weight $-\Lambda$.
\end{remark}

\subsection{Proof of the classification} The conjugate-linear anti-involutions $\omega_{+}$ and $\omega_{-}$ differ only by a sign on the odd part. Up to this sign, the idea of the proofs is identical, and we therefore restrict to $\omega_{+}$.

Fix $\Lambda\in P^{++}_{\omega_{+}}$. We associate to $\Lambda$ its atypicality matrix
$A(\Lambda)=\bigl((\Lambda+\rho,\epsilon_i-\delta_j)\bigr)_{i,j}$,
whose entries strictly decrease along rows and columns:
\begin{equation}\label{diagram_fd}
\begin{array}{ccccccccc}
(\Lambda+\rho,\epsilon_1-\delta_1)&>&\cdots&>&(\Lambda+\rho,\epsilon_1-\delta_k)&>&\cdots&>&(\Lambda+\rho,\epsilon_1-\delta_n)\\
\vee&&&&\vee&&&&\vee\\
\vdots&&&&\vdots&&&&\vdots\\
\vee&&&&\vee&&&&\vee\\
(\Lambda+\rho,\epsilon_m-\delta_1)&>&\cdots&>&(\Lambda+\rho,\epsilon_m-\delta_k)&>&\cdots&>&(\Lambda+\rho,\epsilon_m-\delta_n)
\end{array}
\end{equation}
A necessary condition for unitarity is $(\Lambda+\rho,\alpha)>0$ whenever $L_0(\Lambda-\alpha)$ is a non-trivial $\even$-constituent of $L(\Lambda)$. Hence the ordering reduces the analysis to the last row. By Proposition~\ref{prop::unitarity_conditions}, a zero there can occur only at an index $k$ with $\mu^k=\mu^n$. Set
$
k_0\coloneqq \min\{k:\mu^k=\mu^n\}.
$
Thus $k\geq k_0$. We first treat the boundary cases
$(\Lambda+\rho,\epsilon_m-\delta_n)>0$ and
$(\Lambda+\rho,\epsilon_m-\delta_{k_0})<0$. The remaining cases are
$
(\Lambda+\rho,\epsilon_m-\delta_{k_0})\geq0
$
and
$
(\Lambda+\rho,\epsilon_m-\delta_n)\leq0.
$
By strict decrease, either $(\Lambda+\rho,\epsilon_m-\delta_k)=0$ for a unique $k_0\leq k\leq n$, in which case Proposition~\ref{prop::continuous_parameter} and Corollary~\ref{cor::main_fd} establish unitarity, or no such zero occurs, in which case Proposition~\ref{prop::unitarity_conditions} excludes unitarity.

\begin{lemma}
If $(\Lambda+\rho,\epsilon_m-\delta_n)>0$, then
$L(\Lambda)$ is unitarizable.
\end{lemma}

\begin{proof}
By~\eqref{diagram_fd}, the inequality $(\Lambda+\rho,\epsilon_{m}-\delta_{n})>0$ implies
$(\Lambda+\rho,\alpha)>0$
for every $\alpha\in\Delta_{\bar1}^{+}$. The claim therefore follows from Proposition~\ref{prop::continuous_parameter}.
\end{proof}

\begin{remark}
 If $(\Lambda+\rho, \epsilon_{m}-\delta_{n})>0$, then $\Lambda$ is typical. 
\end{remark}

\begin{lemma}
 If $(\Lambda+\rho, \epsilon_{m}-\delta_{k_{0}})<0$, then $L(\Lambda)$ is not unitarizable.
\end{lemma}

\begin{proof}
Assume by contradiction that $L(\Lambda)$ is unitarizable. Then Proposition~\ref{prop::unitarity_conditions} implies that either
$
(\Lambda+\rho,\epsilon_m-\delta_n)>0,
$
or there exists a unique $1\leq a\leq n$ such that
$
(\Lambda+\rho,\epsilon_m-\delta_a)=0,
$
in which case
$
\mu^a=\mu^{a+1}=\cdots=\mu^n.
$
By assumption,
$
(\Lambda+\rho,\epsilon_m-\delta_{k_0})<0,
$
and hence
$
(\Lambda+\rho,\epsilon_m-\delta_n)<0,
$
since
$
\bigl((\Lambda+\rho,\epsilon_m-\delta_a)\bigr)_{a=1}^{n}
$
is strictly decreasing. Thus there exists some $1\leq b\leq n$ such that
$
(\Lambda+\rho,\epsilon_m-\delta_b)=0,
$
with
$
\mu^b=\mu^{b+1}=\cdots=\mu^n.
$
By the minimality of $k_0$, the latter condition implies
$
b\geq k_0.
$
But strict decrease, together with
$
(\Lambda+\rho,\epsilon_m-\delta_{k_0})<0,
$
gives
$
(\Lambda+\rho,\epsilon_m-\delta_b)<0
$
for $b\geq k_0$, a contradiction.
\end{proof}

Together, the two lemmas reduce the study of unitarity to the Dirac inequalities associated with the roots $\epsilon_m-\delta_k$, $k_0\leq k\leq n$, whose values $(\Lambda+\rho,\epsilon_m-\delta_k)$ lie between $(\Lambda+\rho,\epsilon_m-\delta_{k_0})$ and $(\Lambda+\rho,\epsilon_m-\delta_n)$. We start with the non-integer values.

\begin{lemma} If $(\Lambda+\rho, \epsilon_{m}-\delta_{k}) < 0 < (\Lambda+\rho, \epsilon_{m}-\delta_{k-1})$ for $k_{0}\leq k \leq n$, then $L(\Lambda)$ is not unitarizable.
\end{lemma}

\begin{proof}
 If $(\Lambda+\rho,\epsilon_{m}-\delta_{k})<0<(\Lambda+\rho,\epsilon_{m}-\delta_{k-1})$, then $(\Lambda+\rho,\epsilon_{m}-\delta_{k})\neq0$ for all $k=1,\ldots,n$. Yet the condition $(\Lambda+\rho,\epsilon_{m}-\delta_{k})<0$ is incompatible with the unitarity conditions imposed on $\Lambda$, and $L(\Lambda)$ is not unitarizable.
\end{proof}

\begin{lemma} \label{lemm::discrete_points_are_unitarizable_fd}
 Assume $(\Lambda+\rho, \epsilon_{m}-\delta_{k})=0$ for some $k_{0}\leq k \leq n$. Then $L(\Lambda)$ is unitarizable.
\end{lemma}

\begin{proof}
By Proposition~\ref{prop::continuous_parameter}, it suffices to show that for every highest weight $\nu\neq\Lambda$ of a $\even$-composition factor of $L(\Lambda)$, there exists
$
\gamma=\gamma_1+\cdots+\gamma_s,
$
where the $\gamma_i$ are pairwise distinct positive odd roots satisfying
\[
(\Lambda+\rho,\gamma_i)>0
\qquad
\text{for all }i=1,\ldots,s,
\]
such that
\[
(\nu+2\rho_{\bar0},\nu)
=
(\Lambda-\gamma+2\rho_{\bar0},\Lambda-\gamma).
\]

Assume that $\Lambda$ is atypical of degree $r$.
The entries of $A(\Lambda)$ strictly decrease along rows and columns. By assumption,
$
(\Lambda+\rho,\epsilon_m-\delta_k)=0.
$
Thus the zeros of $A(\Lambda)$ occur at positions $(i_{a},j_{a})$, where
$m=i_1>i_2>\cdots>i_r$, $k=j_1<j_2<\cdots<j_r$.
Set
\[
X \coloneqq
\bigcup_{a=1}^{r}
\{\epsilon_{i_a}-\delta_j:j_a\leq j\leq n\}.
\]
Let $L_0(\nu)$ be a $\even$-composition factor of $L(\Lambda)$. By Proposition~\ref{prop::atypical_composition_factors_degree_k} and Remark~\ref{rmk::extend_atypicality}, $L_{0}(\nu)$ is a subquotient of $M_{0}(\Lambda-\Gamma_{S})$ for $S \subset \Delta^{+}_{\bar{1}}\setminus X$. Hence, there exist pairwise distinct $\gamma_i\in S$ such that
$
\gamma=\gamma_1+\cdots+\gamma_s,
$
and 
\[
(\nu+2\rho_{\bar0},\nu)
=
(\Lambda-\gamma+2\rho_{\bar0},\Lambda-\gamma).
\]
By the ordering of $A$ (\emph{cf.}~\eqref{diagram_fd}), every $\alpha\in S$ satisfies
$
(\Lambda+\rho,\alpha)>0.
$
Hence, $(\Lambda+\rho,\gamma_{i})>0$ for all $i=1,\ldots,s$, and the assertion follows from Proposition~\ref{prop::continuous_parameter}.
\end{proof}

Combining the preceding lemmas proves Theorem~\ref{thm::finite-dimensional-classification}. 

\subsection{Explicit parametrization}

We now give an explicit parametrization of the finite-dimensional $\omega_{+}$-unitarizable simple $\gg$-supermodules. Since
$\su(m,0\vert n,0)_{\bar{0}}\cong\su(0,m\vert 0,n)_{\bar{0}}\cong\su(m)\oplus\su(n)\oplus\u(1)$
is compact, every $\Lambda\in P^{++}$ is the highest weight of a finite-dimensional unitarizable simple $\even$-module $L_{0}(\Lambda)$. Such a module is an outer tensor product of simple $\su(m)^{\CC}$-, $\su(n)^{\CC}$-, and $\mathfrak{u}(1)^{\CC}$-modules. The simple $\su(m)^{\CC}$-modules have highest weights of the form
\begin{equation}
(-a_{1},-a_{2},\ldots,-a_{m}\vert 0,\ldots,0),
\end{equation}
where $a_{1},\ldots,a_{m}\in\ZZ_{\geq0}$ and
$0=-a_{1}\geq-a_{2}\geq\cdots\geq-a_{m}$.
Similarly, the simple $\su(n)^{\CC}$-modules have highest weights of the form
\begin{equation}
(0,\ldots,0\vert b_{1},\ldots,b_{n}),
\end{equation}
where $b_{1},\ldots,b_{n}\in\ZZ_{\geq0}$ and
$b_{1}\geq\cdots\geq b_{n}=0$.
Since $\mathfrak{u}(1)$ is abelian, its simple modules are one-dimensional and, by unitarity, parametrized by a real number $x_{0}$. Thus
\begin{equation} \label{eq::general_form_Lambda_fd}
\Lambda
=
\Lambda_{0}
+\tfrac{x_{0}}{2}(1,\ldots,1\vert 1,\ldots,1),
\qquad
\Lambda_{0}
=
(0,-a_{2},\ldots,-a_{m}\vert b_{1},\ldots,b_{n-1},0),
\qquad
x_{0}\in\RR.
\end{equation}

Theorem~\ref{thm::finite-dimensional-classification} therefore admits the following explicit reformulation.

\begin{theorem}
Fix $\Lambda_{0}$ as in \eqref{eq::general_form_Lambda_fd}, and let
$
k_{0}\coloneqq\min\{k:b_{k}=0\}
$ and $i_{0}\coloneqq \max\{i : a_{i}=0\}$.
\begin{enumerate}
\item[(a)]
Up to parity, the finite-dimensional $\omega_{+}$-unitarizable simple $\gg$-supermodules are precisely $L(\Lambda(x))$, where $\Lambda(x)=\Lambda_{0}+\tfrac{x}{2}(1,\ldots,1\vert 1,\ldots,1)$ and 
\begin{equation*}
x\in\{a_{m}+k_{0}-1,\ldots,a_{m}+n-1\}\sqcup(a_{m}+n-1,\infty).
\end{equation*}
\item[(b)] Up to parity, the finite-dimensional $\omega_{-}$-unitarizable simple $\gg$-supermodules are precisely $L(\Lambda(x))$, where $\Lambda(x)=\Lambda_{0}+\tfrac{x}{2}(1,\ldots,1\vert 1,\ldots,1)$ and 
\begin{equation*}
x\in (-\infty, -b_{1}-m+1) \sqcup \{-b_{1}-m+1, \ldots, -b_{1}-m+i_{0}\}.
\end{equation*}
\end{enumerate}
\end{theorem}

\section{Infinite-dimensional classification}\label{sec::ifd}
\subsection{Statement of the classification}
We classify the infinite-dimensional unitarizable simple $\mathfrak g$-supermodules. By Corollary~\ref{cor::finite_and_infinite_dimensionality}, these occur precisely when $pq\neq0$. By Theorem~\ref{thm::HW_property_M}, the relevant real forms are $\su(p,q\vert 0,n)$ and $\su(p,q\vert n,0)$, corresponding to $\omega_{(-,+)}$ and $\omega_{(+,-)}$, respectively. With respect to the $\kk$-decomposition and the general matrix form~\eqref{eq::general_form}, these anti-involutions are given by
\begin{equation}
\omega_{(-,+)}(X)=
\left(\begin{array}{@{}c|c@{}}
\begin{matrix}
a^{\dagger} & -c^{\dagger} \\ -b^{\dagger} & d^{\dagger}
\end{matrix}
&
\begin{matrix}
-Q_{1}^{\dagger} \\ Q_{2}^{\dagger}
\end{matrix} \\
\hline
\begin{matrix}
-P_{1}^{\dagger} & P_{2}^{\dagger}
\end{matrix}
&
\begin{matrix}
E^{\dagger}
\end{matrix}
\end{array}\right),\qquad
\omega_{(+,-)}(X)=
\left(\begin{array}{@{}c|c@{}}
\begin{matrix}
a^{\dagger} & -c^{\dagger} \\ -b^{\dagger} & d^{\dagger}
\end{matrix}
&
\begin{matrix}
Q_{1}^{\dagger} \\ -Q_{2}^{\dagger}
\end{matrix} \\
\hline
\begin{matrix}
P_{1}^{\dagger} & -P_{2}^{\dagger}
\end{matrix}
&
\begin{matrix}
E^{\dagger}
\end{matrix}
\end{array}\right).
\end{equation}
As explained in Theorem~\ref{thm::HW_property_M}, only the real forms $\su(p,q\vert 0,n)$ admit non-trivial unitarizable highest weight supermodules, and we fix these real forms in what follows. Moreover, as explained in Section~\ref{subsubsec::unitarizable_even_HWM}, the corresponding highest weight supermodules may be defined with respect to the standard positive system $\Delta_{\st}^{+}$ or the non-standard positive system $\Delta_{\nst}^{+}$ associated with $pq$. We fix the non-standard positive system and denote it simply by $\Delta^{+}$. The analogous result for the standard positive system is given in Section~\ref{subsec::standard_positive_system_ifd}.

Any unitarizable highest weight supermodule $L(\Lambda)$ decomposes into a finite direct sum of unitarizable highest weight $\even$-modules. In particular, $\Lambda$ is the highest weight of a unitarizable $\even$-module. By Lemma~\ref{lemm::structure_even_module}, this module is the outer tensor product of a unitarizable highest weight $\su(p,q)$-module, a finite-dimensional unitarizable $\su(n)$-module, and a one-dimensional module for the center of $\even$. Among these, the following exhaust the possibilities that extend to unitarizable $\gg$-supermodules.

\begin{theorem}\label{thm::full_set_ifd}
Let $\Lambda\in\hh^{\ast}$ be the highest weight of a unitarizable highest weight $\even$-module. The unitarizable highest weight $\gg$-supermodules are precisely $L(\Lambda)$ and $\Pi L(\Lambda)$ such that one of the following conditions holds:
\begin{enumerate}
\item[(1)] $(\Lambda+\rho,-\epsilon_m+\delta_n)<0$, and $(\Lambda+\rho,\epsilon_i-\delta_1)=0$ and $(\Lambda,\epsilon_1-\epsilon_i)=0$ for some $i\in\{1,\ldots,p\}$;
\item[(2)] $(\Lambda+\rho,-\epsilon_j+\delta_n)=0$ and $(\Lambda,\epsilon_j-\epsilon_m)=0$ for some $j\in\{p+1,\ldots,m\}$, and $(\Lambda+\rho,\epsilon_i-\delta_1)=0$ and $(\Lambda,\epsilon_1-\epsilon_i)=0$ for some $i\in\{1,\ldots,p\}$;
\item[(3)] $(\Lambda+\rho,-\epsilon_j+\delta_n)=0$ and $(\Lambda,\epsilon_j-\epsilon_m)=0$ for some $j\in\{p+1,\ldots,m\}$, and $(\Lambda+\rho,\epsilon_1-\delta_1)<0$;
\item[(4)] $(\Lambda+\rho,-\epsilon_m+\delta_n)<0$ and $(\Lambda+\rho,\epsilon_1-\delta_1)<0$.
\end{enumerate}
\end{theorem}

\begin{remark}
As explained in Remark~\ref{rmk::symmetry_HW_LW}, the $\omega_{(-,+)}$-unitarizable highest weight supermodules of highest weight $\Lambda$ are in one-to-one correspondence with the $\omega_{(+,-)}$-unitarizable lowest weight supermodules of lowest weight $-\Lambda$. Thus the classification above also gives the full set of unitarizable lowest weight supermodules and hence, by Theorem~\ref{thm::HW_property_M}, the full set of unitarizable simple $\gg$-supermodules.
\end{remark}

\subsection{Proof of the classification}
Let $\Lambda$ be the highest weight of a unitarizable $\even$-module (\emph{cf.}~Section~\ref{subsubsec::unitarizable_even_HWM}). In particular, $\Lambda$ is $\Delta_{c}^{+}$-dominant integral (see~\eqref{eq::dominant_integral}). According to the $\kk$-decomposition of the positive odd root spaces, we decompose $\Delta_{\bar{1}}^{+}$ into two $\kk^{\CC}$-stable subsets:
\begin{equation}
A\coloneqq
\{\epsilon_i-\delta_a:1\leq i\leq p,\ 1\leq a\leq n\},
\qquad
B\coloneqq
\{-\epsilon_i+\delta_a:p<i\leq m,\ 1\leq a\leq n\}.
\end{equation}
By $\Delta_{c}^{+}$-dominance of $\Lambda$, we associate to it the $p$-block atypicality matrix $A(\Lambda)$ and the $q$-block atypicality matrix $B(\Lambda)$, where
\begin{equation}
 A(\Lambda)\coloneqq\bigl((\Lambda+\rho,\epsilon_i-\delta_a)\bigr)_{1\leq i\leq p,\ 1\leq a\leq n}, \qquad B(\Lambda)\coloneqq\bigl((\Lambda+\rho,-\epsilon_i+\delta_a)\bigr)_{p<i\leq m,\ 1\leq a\leq n}.
\end{equation}
The entries of $A(\Lambda)$ strictly decrease along rows and columns, whereas the entries of $B(\Lambda)$ strictly increase along rows and columns. Concretely,
{\small
\setlength{\arraycolsep}{3pt}
\renewcommand{\arraystretch}{0.85}
\begin{equation}\label{eq::diagrams_ifd}
\begin{array}{ccccc}
(\Lambda+\rho,\epsilon_1-\delta_1) & > & \cdots & > & (\Lambda+\rho,\epsilon_1-\delta_n)\\
\vee &&&& \vee\\
\vdots &&&& \vdots\\
\vee &&&& \vee\\
(\Lambda+\rho,\epsilon_i-\delta_1) & > & \cdots & > & (\Lambda+\rho,\epsilon_i-\delta_n)\\
\vee &&&& \vee\\
\vdots &&&& \vdots\\
\vee &&&& \vee\\
(\Lambda+\rho,\epsilon_p-\delta_1) & > & \cdots & > & (\Lambda+\rho,\epsilon_p-\delta_n)
\end{array}, \qquad
\begin{array}{ccccc}
(\Lambda+\rho,-\epsilon_m+\delta_n) & > & \cdots & > & (\Lambda+\rho,-\epsilon_m+\delta_1)\\
\vee &&&& \vee\\
\vdots &&&& \vdots\\
\vee &&&& \vee\\
(\Lambda+\rho,-\epsilon_{m-j}+\delta_n) & > & \cdots & > & (\Lambda+\rho,-\epsilon_{m-j}+\delta_1)\\
\vee &&&& \vee\\
\vdots &&&& \vdots\\
\vee &&&& \vee\\
(\Lambda+\rho,-\epsilon_{p+1}+\delta_n) & > & \cdots & > & (\Lambda+\rho,-\epsilon_{p+1}+\delta_1)
\end{array}
\end{equation}
} 

The discussion is analogous to the finite-dimensional case, except that we now consider $A(\Lambda)$ and $B(\Lambda)$ simultaneously. A necessary condition for unitarity is
$(\Lambda+\rho,\alpha)<0$ whenever $L_{0}(\Lambda-\alpha)$ is a non-trivial $\even$-constituent of $L(\Lambda)$. The ordering therefore reduces the analysis to the column of $A(\Lambda)$ corresponding to $\delta_{1}$ and the column of $B(\Lambda)$ corresponding to $\delta_{n}$.

By Proposition~\ref{prop::unitarity_conditions}, a zero in the relevant column of $A(\Lambda)$ can occur only at an index $i$ such that
$(\Lambda+\rho,\epsilon_i-\delta_1)=0$
and $\lambda^i=\lambda^1$. Set
$
i_0\coloneqq\max\{i:\lambda^i=\lambda^1\}.
$
Similarly, a zero in the relevant column of $B(\Lambda)$ can occur only at an index $m-j$ such that
$(\Lambda+\rho,-\epsilon_{m-j}+\delta_n)=0$
and $\lambda^{m-j}=\lambda^m$. Set
$
j_0\coloneqq\max\{j:\lambda^{m-j}=\lambda^m\}.
$

We first treat the boundary cases
$
(\Lambda+\rho,\epsilon_1-\delta_1)<0
$
and
$
(\Lambda+\rho,\epsilon_{i_0}-\delta_1)>0,
$
for $A(\Lambda)$, and
$
(\Lambda+\rho,-\epsilon_m+\delta_n)<0
$
and
$
(\Lambda+\rho,-\epsilon_{m-j_0}+\delta_n)>0,
$
for $B(\Lambda)$. In the first case of each pair, all relevant entries have the required sign, whereas in the second no permissible zero can occur.

It remains to consider
$
(\Lambda+\rho,\epsilon_1-\delta_1)\geq0
$
and
$
(\Lambda+\rho,\epsilon_{i_0}-\delta_1)\leq0,
$
together with
$
(\Lambda+\rho,-\epsilon_m+\delta_n)\geq0
$
and
$
(\Lambda+\rho,-\epsilon_{m-j_0}+\delta_n)\leq0.
$
By strict monotonicity, each block contains at most one zero. If the corresponding zero occurs at an admissible index, Proposition~\ref{prop::continuous_parameter} and Corollary~\ref{cor::main_fd} establish unitarity. If no such zero occurs, Proposition~\ref{prop::unitarity_conditions} excludes unitarity.

\begin{lemma}\label{lemm::unitarity_ifd}
If $(\Lambda+\rho,\epsilon_{1}-\delta_{1})<0$ and $(\Lambda+\rho,-\epsilon_{m}+\delta_{n})<0$, then $L(\Lambda)$ is unitarizable.
\end{lemma}
\begin{proof}
By~\eqref{eq::diagrams_ifd}, the assumptions imply
$
(\Lambda+\rho,\gamma)<0
$
for every $\gamma\in\Delta_{\bar{1}}^{+}$. The statement follows from Proposition~\ref{prop::continuous_parameter}.
\end{proof}

\begin{lemma}\label{lemm::non_unitarity_ifd}
If either $(\Lambda+\rho,\epsilon_{i_{0}}-\delta_{1})>0$ or $(\Lambda+\rho,-\epsilon_{m-j_{0}}+\delta_{n})>0$, then $L(\Lambda)$ is not unitarizable.
\end{lemma}

\begin{proof}
Assume first that $(\Lambda+\rho,\epsilon_{i_{0}}-\delta_{1})>0$. Since $(\Lambda+\rho,\epsilon_i-\delta_1)$ strictly decreases with $i$ for $1\leq i\leq p$, one has $(\Lambda+\rho,\epsilon_i-\delta_1)>0$ for all $1\leq i\leq i_0$. Suppose that $L(\Lambda)$ is unitarizable. Since $(\Lambda+\rho,\epsilon_1-\delta_1)<0$ is impossible, Proposition~\ref{prop::unitarity_conditions}(b)(ii) gives some $1\leq i\leq p$ such that $(\Lambda+\rho,\epsilon_i-\delta_1)=0$ and $\lambda^1=\cdots=\lambda^i$. The latter implies $i\leq i_0$ by the definition of $i_0$, contradicting $(\Lambda+\rho,\epsilon_i-\delta_1)>0$ for all $i\leq i_0$.

Assume now that $(\Lambda+\rho,-\epsilon_{m-j_0}+\delta_n)>0$. Since $(\Lambda+\rho,-\epsilon_i+\delta_n)$ strictly increases with $i$ for $p+1\leq i\leq m$, one has $(\Lambda+\rho,-\epsilon_i+\delta_n)>0$ for all $m-j_0\leq i\leq m$. If $L(\Lambda)$ were unitarizable, then $(\Lambda+\rho,-\epsilon_m+\delta_n)<0$ would be impossible, so Proposition~\ref{prop::unitarity_conditions}(b)(iii) gives some $p+1\leq i\leq m$ such that $(\Lambda+\rho,-\epsilon_i+\delta_n)=0$ and $\lambda^i=\cdots=\lambda^m$. The latter implies $i\geq m-j_0$ by the definition of $j_0$, contradicting $(\Lambda+\rho,-\epsilon_i+\delta_n)>0$ for all $i\geq m-j_0$. Thus $L(\Lambda)$ is not unitarizable in either case.
\end{proof}

This reduces the analysis to the two intervals with endpoints
$(\Lambda+\rho,\epsilon_{i_{0}}-\delta_{1})$ and $(\Lambda+\rho,\epsilon_{1}-\delta_{1})$, and
$(\Lambda+\rho,-\epsilon_{m}+\delta_{n})$ and $(\Lambda+\rho,-\epsilon_{m-j_{0}}+\delta_{n})$, respectively. Inside these intervals, the admissible zeros occur at the unit-spaced discrete points $(\Lambda+\rho,\epsilon_i-\delta_1)$, $1\leq i\leq i_0$, and $(\Lambda+\rho,-\epsilon_{m-j}+\delta_n)$, $0\leq j\leq j_0$. We will show that $L(\Lambda)$ is unitarizable precisely when, for each block, either the relevant boundary value is negative or one of these admissible discrete points is zero.

\begin{lemma}
If $(\Lambda+\rho,\epsilon_{i+1}-\delta_{1})<0<(\Lambda+\rho,\epsilon_{i}-\delta_{1})$ or $(\Lambda+\rho,-\epsilon_{m-j}+\delta_{n})<0<(\Lambda+\rho,-\epsilon_{m-j+1}+\delta_{n})$ for some $1\leq i<i_{0}$ or $1\leq j\leq j_{0}$, then $L(\Lambda)$ is not unitarizable.
\end{lemma}
\begin{proof}
Assume that $L(\Lambda)$ is unitarizable. Suppose first that
$
(\Lambda+\rho,\epsilon_{i+1}-\delta_{1})<0<(\Lambda+\rho,\epsilon_i-\delta_{1})
$
for some $1\leq i<i_0$. Since the values
$
(\Lambda+\rho,\epsilon_k-\delta_1)
$
strictly decrease with $k$, it follows that
$
(\Lambda+\rho,\epsilon_1-\delta_1)>0
$
and
$
(\Lambda+\rho,\epsilon_k-\delta_1)\neq0
$
for every $1\leq k\leq p$. This contradicts Proposition~\ref{prop::unitarity_conditions}(b)(ii), which requires either
$
(\Lambda+\rho,\epsilon_1-\delta_1)<0
$
or
$
(\Lambda+\rho,\epsilon_k-\delta_1)=0
$
for some $1\leq k\leq p$.

Suppose now that
$
(\Lambda+\rho,-\epsilon_{m-j}+\delta_n)<0<
(\Lambda+\rho,-\epsilon_{m-j+1}+\delta_n)
$
for some \(1\leq j\leq j_0\). Since the values
$
(\Lambda+\rho,-\epsilon_k+\delta_n)
$
strictly increase with \(k\), it follows that
$
(\Lambda+\rho,-\epsilon_m+\delta_n)>0
$
and
$
(\Lambda+\rho,-\epsilon_k+\delta_n)\neq0
$
for every \(p+1\leq k\leq m\). This contradicts Proposition~\ref{prop::unitarity_conditions}(b)(iii), which requires either
$
(\Lambda+\rho,-\epsilon_m+\delta_n)<0
$
or
$
(\Lambda+\rho,-\epsilon_k+\delta_n)=0
$
for some \(p+1\leq k\leq m\). Thus \(L(\Lambda)\) is not unitarizable.
\end{proof}

\begin{lemma}
If one of the following conditions holds, then $L(\Lambda)$ is unitarizable:
\begin{enumerate}
\item[(a)] $(\Lambda+\rho,-\epsilon_{m}+\delta_{n})<0$ and $(\Lambda+\rho,\epsilon_{i}-\delta_{1})=0$ for some $1\leq i\leq i_{0}$.
\item[(b)] $(\Lambda+\rho,-\epsilon_{m-j}+\delta_{n})=0$ and $(\Lambda+\rho,\epsilon_{i}-\delta_{1})=0$ for some $0\leq j\leq j_{0}$ and $1\leq i\leq i_{0}$.
\item[(c)] $(\Lambda+\rho,-\epsilon_{m-j}+\delta_{n})=0$ and $(\Lambda+\rho,\epsilon_{1}-\delta_{1})<0$ for some $0\leq j\leq j_{0}$.
\end{enumerate}
\end{lemma}

\begin{proof}
We prove~(a). The proofs of~(b) and~(c) are analogous. By Proposition~\ref{prop::continuous_parameter}, it suffices to show that for every highest weight $\nu\neq\Lambda$ of a $\even$-composition factor of $L(\Lambda)$ there exist pairwise distinct positive odd roots $\gamma_1,\ldots,\gamma_s$ such that, with
$
\gamma=\gamma_1+\cdots+\gamma_s,
$
one has
$
(\Lambda+\rho,\gamma_l)<0
$
for all $l=1,\ldots,s$ and
\[
(\nu+2\rho_{\bar0},\nu)
=
(\Lambda-\gamma+2\rho_{\bar0},\Lambda-\gamma).
\]

Assume~(a). Since
$
(\Lambda+\rho,-\epsilon_m+\delta_n)<0,
$
the ordering in~\eqref{eq::diagrams_ifd} implies
$
(\Lambda+\rho,\beta)<0
$
for every $\beta\in B$. Thus only roots in $A$ with non-negative pairing have to be excluded.

Assume that $\Lambda$ is atypical of degree $r$. Then there exist $r_A,r_B\in\ZZ_{+}$ such that $r=r_A+r_B$, where $A(\Lambda)$ has $r_A$ zeros and $B(\Lambda)$ has $r_B$ zeros. Since
$
(\Lambda+\rho,-\epsilon_m+\delta_n)<0,
$
the ordering in~\eqref{eq::diagrams_ifd} implies that $B(\Lambda)$ has no zeros, and hence $r_A=r$. The entries of $A(\Lambda)$ strictly decrease along rows and columns. By assumption,
$
(\Lambda+\rho,\epsilon_i-\delta_1)=0.
$
Thus the zeros of $A(\Lambda)$ occur at positions $(i_a,j_a)$, where
$
i=i_1>i_2>\cdots>i_r
$
and
$
1=j_1<j_2<\cdots<j_r.
$
Set
\[
X\coloneqq
\bigcup_{a=1}^{r}
\{\epsilon_i-\delta_{j_a}:1\leq i\leq i_a\}.
\]
Let $L_0(\nu)$ be a $\even$-composition factor of $L(\Lambda)$. By Proposition~\ref{prop::action_Casimir_ifd} and its direct generalization to higher atypicality, there exists some subset $S\subseteq \Delta_{\bar{1}}^{+}\setminus X$ such that $L_0(\nu)$ is a subquotient of $M_0(\Lambda-\Gamma_S)$. Hence there exist pairwise distinct $\gamma_1,\ldots,\gamma_s\in S$ such that
$
\gamma=\gamma_1+\cdots+\gamma_s
$
and
\[
(\nu+2\rho_{\bar0},\nu)
=
(\Lambda-\gamma+2\rho_{\bar0},\Lambda-\gamma).
\]
By the ordering of $A(\Lambda)$ (\emph{cf.}~\eqref{eq::diagrams_ifd}), every $\alpha\in S$ satisfies
$
(\Lambda+\rho,\alpha)<0.
$
Hence $(\Lambda+\rho,\gamma_i)<0$ for all $i=1,\ldots,s$, and the assertion follows from Proposition~\ref{prop::continuous_parameter}.
\end{proof}

Combining the preceding lemmas proves Theorem~\ref{thm::full_set_ifd}.

\subsection{Explicit parametrization}\label{subsec::parametrization_ifd}
We now translate Theorem~\ref{thm::full_set_ifd} into an explicit parametrization. Let $L(\Lambda)$ be a simple $\omega$-unitarizable $\gg$-supermodule, where $\omega\coloneqq\omega_{(-,+)}$. Then $\Lambda$ is necessarily the highest weight of a unitarizable $\even$-module $L_{0}(\Lambda)$. Recall
$
\even=\su(p,q)^{\CC}\oplus\su(n)^{\CC}\oplus\u(1)^{\CC}.
$
We first parametrize these unitarizable $\even$-modules and then determine which of them extend to unitarizable $\gg$-supermodules.

Write
$
\Lambda=(\lambda^1,\ldots,\lambda^m\vert\mu^1,\ldots,\mu^n),
$
modulo shifts by $(1,\ldots,1\vert{-}1,\ldots,-1)$. Let $\kk\coloneqq\su(p)\oplus\su(q)\oplus\u(1)\oplus\su(n)\oplus\u(1)$ be the maximal compact subalgebra of $\su(p,q\vert n)$. If $L_0(\Lambda)$ is unitarizable, then $\Lambda$ is the highest weight of a finite-dimensional simple $\kk^\CC$-module. Hence it is integral and dominant with respect to the positive system induced from $\gg$, that is,
\begin{equation}
\label{inequalities}
\lambda^{p+1}\ge \cdots \ge \lambda^m \ge \lambda^1\ge\cdots\ge\lambda^p,
\qquad
\mu^1\ge \cdots \ge \mu^n.
\end{equation}
Following \cite{jakobsen1994full}, we parametrize these weights by
\begin{multline}
\label{jakobsenpara}
\Lambda=(0,a_{2},\ldots,a_{m-1},0\vert b_{1},\ldots,b_{n-1},0)+\tfrac{\lambda}{2}
(1,\ldots,1,-1,\ldots,-1\vert0,\ldots,0)\\
+\tfrac{x}{2}(1,\ldots,1\vert1,\ldots,1),
\end{multline}
where $a_{p+1}\ge\cdots\ge a_{m-1}\ge0\ge a_2\ge\cdots\ge a_p$, $b_1\ge\cdots\ge b_{n-1}\ge b_n\coloneqq0$, and $x,\lambda\in\RR$. Using the shift-invariance, we impose $\lambda^1+\lambda^m=2\mu^n=:x$ and set $\lambda\coloneqq\lambda^1-\lambda^m$. The conditions on the $a_i$ and $b_k$ are then equivalent to~\eqref{inequalities}.

For fixed $a_i$ and $b_k$, it is known \cite{enright1983classification,jakobsen1994full} that~\eqref{jakobsenpara} is the highest weight of a unitarizable simple highest weight $\even$-module if and only if
\begin{multline}
\label{thisset}
\lambda\in\bigl(-\infty,-m+\max(i_0,j_0)+1\bigr)\\
\cup\bigl\{-m+\max(i_0,j_0)+1,-m+\max(i_0,j_0)+2,\ldots,-m+i_0+j_0\bigr\},
\end{multline}
where $i_0\leq p$ is the largest index for which $a_i=0$, and $1\leq j_{0}\leq q+1$ is the largest integer for which $a_{m-j+1}=0$ (if $a_{p+1}=0$, then $j_0=q$). We also let $k_0$ be the smallest integer such that $b_{k_0}=0$.

Fix now $a_i$, $b_k$, and $\lambda$ satisfying these conditions, and set
\begin{equation}\label{eq::parametrization_HW_ifd}
\Lambda_0\coloneqq(0,a_{2},\ldots,a_{m-1},0\vert b_{1},\ldots,b_{n-1},0)+\tfrac{\lambda}{2}(1,\ldots,1,-1,\ldots,-1\vert0,\ldots,0).
\end{equation}
Then every corresponding unitarizable $\even$-highest weight is
$
\Lambda=\Lambda_0+\tfrac{x}{2}(1,\ldots,1\vert1,\ldots,1).
$
It remains to determine for which values of $x$ the module $L(\Lambda)$ is unitarizable as a $\gg$-supermodule.

The four alternatives in Theorem~\ref{thm::full_set_ifd} separate into two independent conditions. The alternatives involving $(\Lambda+\rho,\epsilon_i-\delta_1)$ give the allowed set $X_A(\Lambda_0)$, while those involving $(\Lambda+\rho,-\epsilon_j+\delta_n)$ give $X_B(\Lambda_0)$. Substituting~\eqref{jakobsenpara} into these conditions gives
\begin{equation}
\begin{aligned}
X_A(\Lambda_0)
&\coloneqq
\left(
-\infty,
-\tfrac{\lambda}{2}-b_1-p+1
\right)
\cup
\left\{
-\tfrac{\lambda}{2}-b_1-p+i
:
1\leq i\leq i_0
\right\},\\
X_B(\Lambda_0)
&\coloneqq
\left(
\tfrac{\lambda}{2}+q-1,
\infty
\right)
\cup
\left\{
\tfrac{\lambda}{2}+j-p-1
:
p+1\leq j\leq m,\ a_j=0
\right\}.
\end{aligned}
\end{equation}
Thus the four cases of Theorem~\ref{thm::full_set_ifd} are precisely the four possible combinations of the continuous or discrete parts of $X_A(\Lambda_0)$ and $X_B(\Lambda_0)$. We therefore obtain the following explicit parametrization.

\begin{theorem}\label{thm::full_set_ifd_parametrization}
The $\omega_{(-,+)}$-unitarizable highest weight $\gg$-supermodules are precisely $L(\Lambda(x))$ and $\Pi L(\Lambda(x))$, where
$\Lambda(x)=\Lambda_{0}+\tfrac{x}{2}(1,\ldots,1\vert1,\ldots,1)$,
with $\Lambda_{0}$ of the form~\eqref{eq::parametrization_HW_ifd}, and
\begin{equation*}
x\in X_A(\Lambda_0)\cap X_B(\Lambda_0).
\end{equation*}
\end{theorem}

\subsection{The standard positive system}\label{subsec::standard_positive_system_ifd} The preceding analysis uses the non-standard positive system $\Delta_{\nst}^{+}$, to which the Dirac criterion applies directly. By Lemma~\ref{lemm::relating_standard_and_non_standard_simple_system}, the standard and non-standard simple systems are related by the sequence of odd reflections
\begin{equation}
\Uppi_{\mathrm{st}}
=
\bigl(r_{\delta_1-\epsilon_m}\cdots r_{\delta_n-\epsilon_m}\bigr)
\cdots
\bigl(r_{\delta_1-\epsilon_{p+1}}\cdots r_{\delta_n-\epsilon_{p+1}}\bigr)
\Uppi_{\mathrm{nst}}.
\end{equation}
Consequently, for every simple highest weight supermodule $L(\Lambda;\Delta_{\nst}^{+})$, there exist $\Lambda'\in\hh^{\ast}$ and $\varepsilon\in\ZZ_{2}$ such that (see~\eqref{eq::relation_Verma_modules_and_systesm})
\begin{equation}
L(\Lambda;\Delta_{\nst}^{+})
\cong
\Pi^{\varepsilon}L(\Lambda';\Delta_{\st}^{+}),
\end{equation}
where the highest weights $\Lambda,\Lambda' \in \hh^{\ast}$ are related by the odd reflection operations. Transporting the positive definite Hermitian form along this identification gives the following.

\begin{proposition}
The highest weight supermodule $L(\Lambda;\Delta_{\nst}^{+})$ is unitarizable if and only if $L(\Lambda';\Delta_{\st}^{+})$ is unitarizable.
\end{proposition}

Thus the classification for $\Delta_{\nst}^{+}$ transfers to the standard positive system and completes the classification.

\begin{proposition}
Let $pq\neq0$. Let $L(\Lambda;\Delta_{\nst}^{+})$ be a unitarizable highest weight $\gg$-supermodule. Then, up to parity, $L(\Lambda;\Delta_{\nst}^{+})$ is unitarily equivalent to $L(\Lambda';\Delta_{\st}^{+})$, where:
\begin{enumerate}
\item[(a)] If $\Lambda$ satisfies either~(1) or~(4) of Theorem~\ref{thm::full_set_ifd}, then
\begin{equation*}
\Lambda'=\Lambda+\sum_{j=p+1}^{m}\sum_{a=1}^{n}(\epsilon_j-\delta_a)
=\Lambda+n\sum_{j=p+1}^{m}\epsilon_j-q\sum_{a=1}^{n}\delta_a.
\end{equation*}
\item[(b)] If $\Lambda$ satisfies either (2) or (3) of Theorem~\ref{thm::full_set_ifd} and $(\Lambda+\rho, -\epsilon_{j}+\delta_{n})=0$, then 
\[
\Lambda'
=\Lambda
+\sum_{k=p+1}^{j-1}\sum_{a=1}^{n}(\epsilon_k-\delta_a)
+\sum_{k=j}^{m}
\sum_{\substack{1\le a\le n\\
\mu^a-\mu^n\ge k-j+1}}
(\epsilon_k-\delta_a) = \Lambda+\sum_{a=1}^{n}
\sum_{k=p+1}^{p+\min\{q,\lambda^m+\mu^a\}}
(\epsilon_k-\delta_a).
\] 
\end{enumerate}
\end{proposition}

\begin{proof}
We transport $\Uppi_{\nst}$ to $\Uppi_{\st}$ using the odd reflections from Lemma~\ref{lemm::relating_standard_and_non_standard_simple_system}. Thus the reflections $r_{\delta_a-\epsilon_j}$ are applied to $\Lambda$ in the order
$
(j,a)=(p+1,n),\ldots,(p+1,1),(p+2,n),\ldots,(m,1).
$
In what follows, write
$
\Lambda=(\lambda^1,\ldots,\lambda^m\vert\mu^1,\ldots,\mu^n)
$
in standard coordinates.

We start proving (a). Both cases~(1) and~(4) of Theorem~\ref{thm::full_set_ifd} satisfy
$
(\Lambda+\rho_{\nst},-\epsilon_m+\delta_n)<0.
$
 The inequality $(\Lambda+\rho_{\nst},-\epsilon_m+\delta_n)<0$ is equivalent to
$
\lambda^m+\mu^n>q-1.
$
Moreover, compact dominance gives $\lambda^j\geq\lambda^m$ for $p<j\leq m$ and $\mu^a\geq\mu^n$ for $1\leq a\leq n$.

Let $\Gamma$ be the current highest weight before applying $r_{\delta_a-\epsilon_j}$. If $(\Gamma,\delta_a-\epsilon_j)\neq0$, then the highest weight changes by
$
\Gamma\mapsto\Gamma+\epsilon_j-\delta_a.
$
Assume inductively that all preceding pairings are nonzero. Before applying $r_{\delta_a-\epsilon_j}$, the $\epsilon_j$- and $\delta_a$-coordinates of $\Gamma$ are respectively
$
\lambda^j+n-a
$
and
$
\mu^a-(j-p-1).
$
Hence
\begin{equation*}
\begin{aligned}
(\Gamma,\delta_a-\epsilon_j)
=-\lambda^j-\mu^a-(n-a)+(j-p-1)\leq-\lambda^m-\mu^n+q-1<0.
\end{aligned}
\end{equation*}
Thus every pairing is nonzero, so every odd reflection changes the highest weight by $\epsilon_j-\delta_a$. After all reflections,
\begin{equation*}
\Lambda'=\Lambda+\sum_{j=p+1}^{m}\sum_{a=1}^{n}(\epsilon_j-\delta_a)
=\Lambda+n\sum_{j=p+1}^{m}\epsilon_j-q\sum_{a=1}^{n}\delta_a.
\end{equation*}
This establishes (a).

For~(b), choose $j\in\{p+1,\ldots,m\}$ as in Theorem~\ref{thm::full_set_ifd}. The assumed equalities and compact dominance give
\[
\lambda^j=\cdots=\lambda^m,\qquad
\lambda^m+\mu^n=j-p-1.
\]
Here we use
$
(\Lambda+\rho_{\nst},-\epsilon_j+\delta_n)=0,
$
$
(\Lambda,\epsilon_j-\epsilon_m)=0,
$
and
$
(\rho_{\nst},-\epsilon_j+\delta_n)=j-p-1.
$
For $p<k<j$, the argument of~(a) gives inductively
\[
\begin{aligned}
(\Gamma,\delta_a-\epsilon_k)
=-\lambda^k-\mu^a-(n-a)+(k-p-1)\le-\lambda^m-\mu^n+(k-p-1)
=k-j<0.
\end{aligned}
\]
Thus all reflections involving $\epsilon_{p+1},\ldots,\epsilon_{j-1}$ change the highest weight. After these reflections, the weight is
\[
\Lambda+\sum_{k=p+1}^{j-1}\sum_{a=1}^{n}(\epsilon_k-\delta_a).
\]
For $k\ge j$, we prove by induction in the prescribed order that
\[
(\Gamma,\delta_a-\epsilon_k)\ne0
\quad\Longleftrightarrow\quad
\mu^a-\mu^n\ge k-j+1.
\]
Assume the assertion for all preceding reflections. Immediately before applying $r_{\delta_a-\epsilon_k}$, the $\delta_a$-coordinate is
\[
\mu^a-(j-p-1)-\min\{k-j,\mu^a-\mu^n\},
\]
and the $\epsilon_k$-coordinate is
\[
\lambda^m+
\#\{b\in\{a+1,\ldots,n\}:\mu^b-\mu^n\ge k-j+1\}.
\]
Using $\lambda^m+\mu^n=j-p-1$, we obtain
\[
\begin{aligned}
(\Gamma,\delta_a-\epsilon_k)
={}&-\max\{\mu^a-\mu^n-(k-j),0\}-\#\{b\in\{a+1,\ldots,n\}:
\mu^b-\mu^n\ge k-j+1\}.
\end{aligned}
\]
If $\mu^a-\mu^n\ge k-j+1$, the first term is strictly negative. Otherwise, integrality and dominance give
\[
\mu^b-\mu^n\le\mu^a-\mu^n\le k-j
\qquad(b>a),
\]
so both terms vanish. This proves the induction.
Each nonzero pairing contributes $\epsilon_k-\delta_a$ to the highest weight. Hence
\[
\Lambda'
=\Lambda
+\sum_{k=p+1}^{j-1}\sum_{a=1}^{n}(\epsilon_k-\delta_a)
+\sum_{k=j}^{m}
\sum_{\substack{1\le a\le n\\
\mu^a-\mu^n\ge k-j+1}}
(\epsilon_k-\delta_a).
\]
For fixed $a$, the contributing indices satisfy
\[
p+1\le k\le
\min\{m,j-1+\mu^a-\mu^n\}
=p+\min\{q,\lambda^m+\mu^a\}.
\]
Therefore
\[
\Lambda'
=\Lambda+\sum_{a=1}^{n}
\sum_{k=p+1}^{p+\min\{q,\lambda^m+\mu^a\}}
(\epsilon_k-\delta_a),
\]
which proves~(b).
\end{proof}

\begin{theorem}
Let $pq\neq0$ and fix the standard positive system. Let $\Lambda$ be the highest weight of a unitarizable highest weight $\even$-module, and set
$
d(\Lambda)\coloneqq\#\{j\in\{p+1,\ldots,m\}:(\Lambda,\epsilon_j-\delta_1)=0\}.
$
Then the unitarizable highest weight $\gg$-supermodules are precisely $L(\Lambda)$ and $\Pi L(\Lambda)$ such that one of the following mutually exclusive conditions holds:
\begin{enumerate}
\item[(1)] $(\Lambda+\rho,\epsilon_m-\delta_n)>0$, and $(\Lambda+\rho,\epsilon_i-\delta_1)=0$ and $(\Lambda,\epsilon_1-\epsilon_i)=0$ for some $i\in\{1,\ldots,p\}$;
\item[(2)] $(\Lambda+\rho,\epsilon_m-\delta_a)=(\Lambda,\delta_a-\delta_n)=0$ for some $a\in\{1,\ldots,n\}$, and $(\Lambda+\rho,\epsilon_i-\delta_1)=d(\Lambda)$ and $(\Lambda,\epsilon_1-\epsilon_i)=0$ for some $i\in\{1,\ldots,p\}$;
\item[(3)] $(\Lambda+\rho,\epsilon_m-\delta_a)=(\Lambda,\delta_a-\delta_n)=0$ for some $a\in\{1,\ldots,n\}$, and $(\Lambda+\rho,\epsilon_1-\delta_1)<d(\Lambda)$;
\item[(4)] $(\Lambda+\rho,\epsilon_m-\delta_n)>0$ and $(\Lambda+\rho,\epsilon_1-\delta_1)<0$.
\end{enumerate}
\end{theorem}

\begin{proof}
Write $\Lambda'$ for the highest weight with respect to $\Delta_{\nst}^{+}$. The preceding proposition gives
\[
(\Lambda+\rho_{\st},\epsilon_i-\delta_1)
=(\Lambda'+\rho_{\nst},\epsilon_i-\delta_1)+d(\Lambda),
\qquad
(\Lambda,\epsilon_1-\epsilon_i)
=(\Lambda',\epsilon_1-\epsilon_i)
\]
for $1\le i\le p$. In cases~(1) and~(4) of Theorem~\ref{thm::full_set_ifd}, every reflection changes the highest weight, so
\[
d(\Lambda)=0,\qquad
\Lambda+\rho_{\st}=\Lambda'+\rho_{\nst}.
\]
These cases therefore give~(1) and~(4) of the present theorem.
In cases~(2) and~(3), let $j$ be the index in Theorem~\ref{thm::full_set_ifd} and write
$\Lambda'=((\lambda')^1,\ldots,(\lambda')^m\vert
(\mu')^1,\ldots,(\mu')^n)$. Set
\[
s=1+\#\{a:(\mu')^a-(\mu')^n\ge m-j+1\}.
\]
One has $1\le s\le n$. The formula in part~(b) of the proposition gives
\[
\lambda^m=(\lambda')^m+s-1,\qquad
\mu^s=\cdots=\mu^n=-(\lambda')^m.
\]
Hence
\[
(\Lambda+\rho_{\st},\epsilon_m-\delta_s)
=(\Lambda,\delta_s-\delta_n)=0.
\]
The first two identities then give precisely conditions~(2) and~(3).
Conversely, reversing the reflections gives the same equivalences of conditions. To apply Theorem~\ref{thm::full_set_ifd}, its $\even$-module hypothesis must also be checked for $\Lambda'$. This follows from the classical $\mathfrak{su}(p,q)$ criterion: when both equalities occur, their indices $i,j$ give
\[
(\lambda')^m-(\lambda')^1
=j-i-1+(\mu')^1-(\mu')^n
\in j-i-1+\mathbb Z_{\ge0}.
\]
When only the equality indexed by $j$ occurs, this difference is $>j-2$; when only the equality indexed by $i$ occurs, it is $>m-i-1$; when neither occurs, it is $>m-2$. These are sufficient bounds for $\even$-module unitarity with the corresponding constant coordinate blocks. Theorem~\ref{thm::full_set_ifd} therefore proves the converse.
Finally, the equality involving $\epsilon_m-\delta_s$ implies
$(\Lambda+\rho_{\st},\epsilon_m-\delta_n)=s-n\le0$.
Moreover, whenever $(\Lambda,\epsilon_1-\epsilon_i)=0$,
\[
(\Lambda+\rho_{\st},\epsilon_1-\delta_1)
=(\Lambda+\rho_{\st},\epsilon_i-\delta_1)+i-1.
\]
These identities prove mutual exclusivity.
\end{proof}

\subsection{The projective case \texorpdfstring{$\mathfrak{psl}(n\vert n)$}{}}
If $m=n$, the special linear Lie superalgebra $\mathfrak{sl}(n\vert n)$ is not simple, as it contains a non-trivial ideal generated by the identity matrix $E_{2n}$. The quotient Lie superalgebra $\mathfrak{psl}(n\vert n)\coloneqq \mathfrak{sl}(n\vert n)/\CC E_{2n}$ is simple and is called the \emph{projective special linear Lie superalgebra}. Our classification result for $p,q\neq0$ provides a complete classification of unitarizable $\mathfrak{psl}(n\vert n)$-supermodules, after imposing the additional condition that $\Lambda=(\lambda^{1},\ldots,\lambda^{n}\vert\mu^{1},\ldots,\mu^{n})$ satisfies $\sum_{i=1}^{n}\lambda^{i}+\sum_{j=1}^{n}\mu^{j}=0$.

\thispagestyle{empty}